\documentclass{amsart}
\usepackage[latin1]{inputenc}
\usepackage[active]{srcltx}
\usepackage{amsfonts,enumerate}

\usepackage[mathscr]{euscript}
\usepackage{epsfig}
\usepackage{latexsym}
\usepackage{psfrag}
\usepackage{multicol}
\usepackage[all]{xy}

\usepackage{amssymb}
\usepackage{amsthm}
\usepackage{amsmath}
\usepackage{amsxtra}
\usepackage{url}

.tfm
.tfm

\theoremstyle{plain}
\newtheorem{prop}{Proposition}[section]
\newtheorem{thm}[prop]{Theorem}
\newtheorem{coro}[prop]{Corollary}

\newtheorem{lemma}[prop]{Lemma}
\newtheorem{conjecture}[prop]{Conjecture}

\newtheorem*{thm*}{Theorem}
\newtheorem{summary}[prop]{Summary}

\theoremstyle{definition}
\newtheorem{defi}[prop]{Definition}

\theoremstyle{remark}

\newtheorem{remark}[prop]{Remark}

\numberwithin{table}{section}

\DeclareMathOperator{\Kfield}{K}
\DeclareMathOperator{\Gal}{Gal}
\DeclareMathOperator{\Vol}{Vol}

\DeclareMathOperator{\ord}{ord}
\DeclareMathOperator{\Spec}{Spec}
\DeclareMathOperator{\Div}{div}

\newcommand{\bfk}{\mathbf k}
\newcommand{\bfz}{\mathbf z}

\newcommand{\bfD}{\mathcal D}

\newcommand{\bfC}{\mathscr C}

\newcommand{\specY}{\mathcal Y}

\newcommand{\bfK}{\mathscr K}

\newcommand{\Sc}{\mathcal{S}}
\newcommand{\Om}{{\mathscr{O}}}

\newcommand{\GL}{{\rm GL}}
\newcommand{\SL}{{\rm SL}}
\newcommand{\PSL}{{\rm PSL}}
\newcommand{\PGL}{{\rm PGL}}
\newcommand{\PGamma}{{\rm P}\Gamma}

\def\ZZ{\mathbb Z}

\def\RR{\mathbb R}

\def\PP{\mathbb P}

\def\CC{\mathbb C}

\DeclareMathOperator{\img}{Im}
\DeclareMathOperator{\Tr}{Tr}

\def\Siegel{\mathfrak H}

\def\<#1>{{\left\langle{#1}\right\rangle}}

\def\Z{{\mathbb Z}}             % integers
\def\Q{{\mathbb Q}}             % rationals
\def\R{{\mathbb R}}             % reals

\def\id#1{{\mathfrak{#1}}}      % an ideal
\def\normid#1{{\norm{\id{#1}}}}
\DeclareMathOperator{\norm}{{\mathscr N}}

\DeclareMathOperator{\trace}{{\mathrm{Tr}}}

\def\?{\ ???\ \immediate\write16{}%
\immediate\write16{Warning: There was still a question mark . . . }%
\immediate\write16{}}

\begin{document}

\title[Hecke/Sturm bounds for HMF over real quadratic fields]{Hecke and Sturm bounds for Hilbert modular forms over real quadratic fields}

\author{Jose Ignacio Burgos Gil}
\address{ICMAT (CSIC-UAM-UCM-UC3), C/ Nicol\'as Cabrera 13-15, 28049
  Madrid, Spain } 
\email{jiburgosgil@gmail.com}
\thanks{JBG was partially supported by grants MTM2009-14163-c02-01 and
 MTM2010-17389.}

\author{Ariel Pacetti}
\address{Departamento de Matem\'atica, Facultad de Ciencias Exactas y Naturales, Universidad de Buenos Aires and IMAS, CONICET, Argentina}
\email{apacetti@dm.uba.ar}
\thanks{AP was partially supported by CONICET PIP 2010-2012 GI and FonCyT BID-PICT 2010-0681.}

\begin{abstract} In this article we give an analogue of Hecke and
  Sturm bounds for Hilbert modular forms over real quadratic fields.
  Let $K$ be a real quadratic field and $\Om_K$ its ring of
  integers. Let $\Gamma$ be a congruence subgroup of $\SL_2(\Om_K)$
  and $M_{(k_1,k_2)}(\Gamma)$ the space of Hilbert modular forms of
  weight $(k_1,k_2)$ for $\Gamma$. The first main result is an
  algorithm to construct a finite set $S$, depending on $K$, $\Gamma$
  and $(k_1,k_2)$, such that if the Fourier expansion coefficients of
  a form $G \in M_{(k_1,k_2)}(\Gamma)$ vanish on the set $S$, then $G$
  is the zero form. The second result corresponds to the same
  statement in the Sturm case, i.e. suppose that all the Fourier
  coefficients of the form $G$ lie in a finite extension of $\Q$, and
  let $\id{p}$ be a prime ideal in such extension, whose norm is
  unramified in $K$; suppose furthermore that the Fourier expansion
  coefficients of $G$ lie in the ideal $\id{p}$ for all the elements in $S$,
  then they all lie in the ideal $\id{p}$.
\end{abstract}

\maketitle

\section*{Introduction}

It is a classical result that the space of modular forms of a fixed
weight and level is finite dimensional. Since modular forms admit a
Fourier expansion, this implies that a few Fourier coefficients should
be enough to determine the form uniquely, but how many coefficients
are needed?

For classical modular forms, this was already known by Hecke (see
\cite{Hecke}, page 811, Satz 1 and Satz 2). Let $\Gamma$ be a
congruence subgroup of $\SL_2(\ZZ)$. Write $\PGamma $ for the image of
$\Gamma $ in $\PSL_{2}(\ZZ)$ and $d=[\PSL_{2}:\PGamma ]$ for the
degree of the map $X(\Gamma )\to X(1)$.  Let $f(z) \in M_{2k}(\Gamma)$
be a weight $2k$ modular form for $\Gamma$ and $f(z) = \sum_{n \ge 0}
a_n(f) q^{n}$ its Fourier expansion at a cusp, where $q=e^{\frac{2\pi i z}{N}}$ is a local uniformizer. Recall that
the order of $f$ at the cusp is defined as
\begin{displaymath}
  \ord(f)=\inf\{n\mid a_{n}(f)\not = 0\}.
\end{displaymath}

\begin{thm*}[Hecke]
Let $f(z)  \in
M_{2k}(\Gamma)$  be a weight $2k$
modular form for $\Gamma$. If $\ord(f)>dk/6$, then $f=0$.
\end{thm*}

Note that this bound is somehow optimal, since for $\SL_2(\ZZ)$, the
number of conditions ``coincides'' (up to 1 depending in the
congruence of the weight 
modulo $12$) with the dimension of the space $M_{2k}(\SL_2(\ZZ))$. 

One can consider the same problem with congruence conditions instead
of vanishing conditions. Let $\Om$ be the ring of integers of a number
field $F$, and $\id{m}$ a maximal ideal of $\Om$. We fix an embedding
$F\subset \CC$. As before, let $f(z) = \sum_{n \ge 0} a_n(f) q^n \in
M_{2k}(\Gamma)$ be a modular form such that $a_n(f) \in \Om$ for all
$n\ge 0$. Then we define
\begin{displaymath}
  \ord_{\id{m}}(f)=\inf\{n \mid a_{n}(f)\not \in \id{m}\},
\end{displaymath}
with the convention $\ord_{\id{m}}(f)=\infty$ if $a_{n}(f)\in \id{m}$
for all $n$.
\begin{thm*}[Sturm] 
If $\ord_{\id{m}}(f)>dk/6$, then $\ord_{\id{m}}(f)=\infty$.
\end{thm*}
The main results of this article are generalizations of both results
to Hilbert modular forms over real quadratic fields. Given a real
quadratic field $K$, a congruence subgroup $\Gamma$ and weights
$(k_1,k_2)$, we give an algorithm to construct a constant $a$,
depending on invariants of the field as well as the congruence
subgroup and the weights, such that if a Hilbert modular form for
$\Gamma$ of weight $(k_1,k_2)$ has order of vanishing at a cusp
greater than $a$, then its order of vanishing is $\infty$. We
furthermore, relate the order of vanishing with Fourier expansions,
i.e. we construct a finite set of elements which determine a form
uniquely by looking at its Fourier coefficients on this set.  The main
idea of the proof is to mimic the geometric proof for classical
modular forms (which is given in the first section) in these new
setting. For that purpose we need something which looks like the
degree function, whose role will be played by a numerical effective
divisor (NEF) in our surface whose intersection number with the cusp
resolutions is non-zero. The first application of this method was
presented as an appendix in~\cite{DPS}, where using a similar approach
we gave a Sturm/Hecke bound for $K=\Q(\sqrt{5})$, level
$\Gamma_0(12\sqrt{5})$ and parallel weight $2$.

The article is organized as follows: in the first section we give a
proof of the classical Hecke and Sturm theorems that although is not
the original one, it is standard and generalizable to our setting.

In the second section, we recall the main properties and definitions
of Hilbert modular surfaces, their desingularization and their
classification. We also give criterions to decide for a particular
level, if the given surface is in minimal model and is of general
type.

In the third section, we recall the main properties of Hilbert modular
forms over real quadratic fields, and we prove the relation between
the order of vanishing at a cusp and vanishing of Fourier expansion coefficients.

In the fourth section we state and prove the analogue of Hecke's
Theorem for parallel weight $2$ Hilbert modular forms over real
quadratic fields with maximal level structure. The statement is self
contained (so there is no need to read the previous sections to
understand the statement) but the proof uses the discussions of the
previous sections.

In the fifth section we adapt the proof of the previous section to
prove the analogue of Sturm's Theorem for parallel weight $2$ Hilbert
modular forms 
over real quadratic fields with maximal level structure. The statement
is the same in both cases, but the proof in this case uses the integral
structure of the modular surfaces.

The sixth section contains statements and proofs for arbitrary weights
and levels and some remarks about its effectiveness. The last section
contains examples of the method as well as some tables comparing the
dimension of the spaces involved and the number of Fourier
coefficients needed using our results in each case.

We end the article with two appendices, the first one explains the
cusp desingularization algorithm needed for the Hecke and Sturm
theorems and the second one treats the real quadratic fields not covered
by the method described in the previous sections.

\section{A geometric proof  of Hecke and Sturm theorems}
We want to sketch well known proofs of Hecke and Sturm theorems that,
although are different than the original proofs of Hecke and Sturm,
are generalizable to higher dimensions.

Recall the following facts about divisors. Let $C$ be a curve defined
over a field $F$.
\begin{itemize}
\item The group of divisors of $C$ is the free abelian group
  generated by the closed points of $C$ (so elements are of the form
  $D = \sum n_p [P]$).  
\item The divisor $D$ is called effective if $n_P \ge 0$ for all $P$.
\item Let $\Kfield(C)$ be the field of rational functions on $C$. To a
  divisor $D$ of $C$ we can associate the (finite 
  dimensional) vector space 
\[
{\mathcal{L}}(D) = \{f\in \Kfield(C) \mid \Div(f) \ge -D\} \cup \{0\}.
\]
\item The degree of the  divisor $D=\sum n_{p}[P]$ is defined as
  \begin{displaymath}
    \deg(D)=\sum n_{P}[k(P):F],
  \end{displaymath}
where $k(P)$ is the residue field at $P$.
If $\deg(D) <0$ then ${\mathcal{L}}(D) = \{0\}$.
\end{itemize}

We start by proving the Hecke bound for $\Gamma =\SL_2(\ZZ)$. Choose
$N\ge 3$. Then the modular curve $Y(\Gamma (N))$ is a smooth compact
complex curve.

Let $g$ be
the genus of $Y(\Gamma(N))$ and $c$ the number of cusps. Denote the
different cusps of $Y(\Gamma(N) )$ by $\sigma _{1},\dots,\sigma _{c}$.

Choose a rational differential form $\omega $ in $Y(\Gamma(N))$ and
let $K=\Div(\omega )$ be the corresponding canonical divisor. 
If $f(z) \in M_{2k}(\SL_{2}(\ZZ))$, then 
$\frac{f (dz)^{\otimes 
    k}}{\omega ^{\otimes k}}$ is a well defined rational function that
belongs to the space
\[
\mathcal{L}\left(k(K +\sum_{i=1}^{c}[\sigma _{i}])
  -\ord(f)N\sum_{i=1}^{c}[\sigma _{i}]\right).
\]
The degree of the divisor $D:=k(K +\sum_{i=1}^{c}[\sigma _{i}])
-\ord(f)N\sum_{i=1}^{c}[\sigma _{i}]$ is given by
$k(2g-2+c)-\ord(f)Nc$. Since, $2g-2+c=Nc/6$ \, and, by hypothesis,
$\ord(f)>k/6$, we conclude that 
$\deg(D)<0$, hence $f=0$.

We next prove the Sturm bound for $\Gamma =\SL_{2}(\ZZ)$.
Let $p =
\id{m}\cap \ZZ$. Choose $N$ such 
that $N\ge 3$ and $p \nmid N$. Let $\zeta_{N}$ be a primitive $N$-th
root of unity. Let $F'=F[\zeta_{N}]$ and
$\Om'$ the ring of integers of $F'$. Since $\Om'$ is integral over
$\Om$, there exists a prime ideal $\id{m}'$ of $\Om'$ such that $\id{m'}\cap
\Om=\id{m}$. Hence, if $\ord_{\id{m}'}(f)=\infty$, then
$\ord_{\id{m}}(f)=\infty$. Thus, replacing $\Om$ by $\Om'$, we may
assume without loss of 
generality that $\zeta_{N}\in \Om$. 

Since $\zeta_{N}\in \Om$, the curve
$Y(\Gamma(N) )$ has an integral smooth 
model over $S=\Spec (\mathcal{O}[1/N])$, denoted $\specY(\Gamma(N)
)$ and each cusp $\sigma_{i}$ of  $Y(\Gamma(N) )$ determines a section
$\overline \sigma _{i}\colon S\to \specY(\Gamma(N)
)$, hence a horizontal divisor, also denoted by  $\overline \sigma
_{i}$. Let $\bfK$ be
the relative canonical divisor of $\specY(\Gamma(N) )/S$. The
$q$-expansion principle \cite[Corollary 1.6.2]{Katz}, implies that $f$
determines a section, also denoted $f$,  of $\mathcal{O}_{\specY(\Gamma(N)
)}(k(\bfK+\sum
\overline \sigma 
_{i}))$. Let $\specY(\Gamma(N) )_{\id{m}}$ be the fiber of
$\specY(\Gamma (N)
)$ over $\id{m}$. It is a smooth curve over the field $k(\id{m})$. The
restriction of $\bfK$ to this curve agrees with its canonical
divisor, denoted $\bfK_{\id{m}}$. We denote by $\overline \sigma
_{i,\id{m}}$ the restriction of the horizontal divisor $\overline
\sigma _{i}$ to $\specY(\Gamma )_{\id{m}}$. Since $\overline
\sigma _{i}$ is given by a section,
the divisor $\overline \sigma _{i,\id{m}}$ is prime and satisfies $k(\overline
\sigma _{i,\id{m}})=k(\id{m})$. The hypothesis of the
theorem imply that the restriction of 
$f$ to $\specY(\Gamma )_{\id{m}}$ determines an element of
\[
\mathcal{L}\left(k(\bfK_{\id{m}} +\sum_{i=1}^{c}[\overline\sigma
  _{i,\id{m}}]) -\ord_{\id{m}}(f)N\sum_{i=1}^{c}[\overline\sigma
  _{1,\id{m}}]\right). 
\]
By the same argument as before this restriction is zero, thus
$\ord_{\id{m}}(f)=\infty$. 

Let now $\Gamma $ be a congruence subgroup. Any element $\gamma \in
\PSL_{2}(\ZZ)$ acts on $M_{2k}(\Gamma )$ by $f\mapsto f|_{2k}[\gamma
]$ and the elements $\gamma \in \PGamma $ act trivially. Let $f(z)$ be
as in Hecke's Theorem. Write
\begin{displaymath}
  g=\prod_{\gamma \in \PGamma \backslash \PSL_{2}(\ZZ)}f|_{2k}[\gamma]. 
\end{displaymath}
Then $g\in M_{2kd}(\SL_{2}(\ZZ))$ and $\ord(g)\ge\ord(f)$. Thus, if
$\ord(f)>kd/6$ we deduce that $g=0$ and a fortiori $f=0$. The same
argument proves the Sturm bound. 

\section{Hilbert modular surfaces}
\subsection{Basic definitions and notations}
Let $D>0$ be a fundamental discriminant, $K=\Q(\sqrt{D})$ the real
quadratic field of discriminant $D$ (which we think of inside the real
numbers), $\Om_K$ its ring of integers and $\delta$ the different of
$\Om_K$. If $\alpha \in K$, we denote by $\alpha'$ its conjugate under
the action of the generator of $\Gal(K/\Q)$. An element $\alpha \in
K$ is called \emph{totally positive} (and denoted $\alpha \gg 0$) if
$\alpha >0$ and $\alpha'>0$.

If $\id{a} \subset K$ is a fractional ideal, we denote by
$\Gamma(\Om_K,\id{a})$ the image in $\PGL_2^+(K)$ of the group
\[
\SL_2(\Om_K,\id{a})=\left\{m \in\left(\begin{array}{cc}
\Om_K & \id{a}^{-1}\\
\id{a} & \Om_K 
\end{array}
\right) \; : \; \det(m)=1\right \}.
\]

If $\id{c}$ is an integral ideal in $K$, we denote by $\Gamma(\id{c},\id{a})$
the image in $\Gamma(\Om_K,\id{a})$ of the group
\[
\left\{ \left(\begin{array} {cc}
\alpha & \beta\\
\gamma & \delta \end{array}
\right) \in \SL_2(\Om_K,\id{a}) \; : \; \alpha \equiv \delta \equiv 1 \pmod{\id{c}}, \beta \in \id{c}\id{a}^{-1},\gamma \in \id{c} \id{a}
\right \}.
\]

A \emph{congruence subgroup} $\Gamma_{\id{a}} \subset \Gamma(\Om_K,\id{a})$ is
a subgroup which contains $\Gamma(\id{c},\id{a})$ for some ideal $\id{c}$.

\medskip

The group $\GL_2^+(K)$ acts on $\Siegel^2$ via 
\[
\left(\begin{array}{cc}
\alpha & \beta\\
\gamma & \delta \end{array} \right) (z_1,z_2) = \left(\frac{\alpha z_1+\beta}{\gamma z_1+\delta},\frac{\alpha' z_2 +\beta'}{\gamma'z_2+\delta'}\right).
\]
Since the center acts trivially, we can consider the action of
$\PGL_2^+(K)$. 

If $\Gamma$ is a congruence subgroup, the quotient
$\Gamma \backslash \Siegel^2$ is a quasi-projective variety with at
most quotient singularities. The
Baily-Borel compactification of such quotient, which we denote
$X_\Gamma$, is obtained as in the classical case by adding the cusps
$\PP^1(K)$ to the product of two copies of the upper half plane, i.e. $X_\Gamma =
\Gamma\backslash (\Siegel^2 \cup \PP^1(K))$. It is a projective variety.

We denote by $Y_\Gamma$ the minimal desingularization of $X_\Gamma$
and by $Z_\Gamma$ the surface obtained by resolving only the cusp
singularities of $X_\Gamma$ which we study in the next sections.

\subsection{On Cusp Resolution} \label{sec:cusp-resolution}
We briefly recall the cusp desingularization at infinity. 
For this section we follow closely the exposition of \cite{VDG}.  If
$M$ is a lattice in $K$, we denote by $U_M^+$ the group (under
multiplication) of totally positive elements $\epsilon \in K$ such
that $\epsilon M = M$. Let $V \subset
U_M^+$ be a subgroup of finite index. We define
\[
G(M,V) = \left\{\left(\begin{array}{cc} \epsilon & m\\
0 & 1 \end{array} \right) \,: \, \epsilon \in V \, , \, m \in M\right\}=M \rtimes V.
\]

If we denote by $U_{\Om_{K},\id{c}}$ the set of units
of $\Om_{K}$ that are congruent to $1$ modulo $\id{c}$, for the
particular congruence subgroups we will consider, we have the
following result.
\begin{lemma} The isotropy group of
the cusp corresponding to  $(\alpha:\beta) \in
\PP^1(K)$ in $\Gamma(\id{c},\id{a})$ is conjugate to the image in
$\PGL_2^+(K)$ of
$$G(\id{a}^{-1}\id{b}^{-2}\id{c},U_{\Om_{K},\id{c}}^2),$$ where $\id{b}
= \alpha \Om_K + \beta \id{a}^{-1}$.
\end{lemma}
\begin{proof}
See the proof of Lemma 5.2 in \cite{VDG} (p. 78).
\end{proof}
In particular the isotropy of the infinity cusp (corresponding to
$(1:0)$ which we denote $\infty$) equals  the image in $\PGL_2^+(K)$
of the group $G(\id{a}^{-1}\id{c},U_{\Om_K,\id{c}}^2)$.
\medskip

We consider the group $\Gamma (\id{c},\id{a})$.
Let $M=\id{a}^{-1}\id{c}\subset K\subset\R$ be the lattice corresponding to the
stabilizer of the 
$\infty$-cusp. It acts on $\CC^2$ by translation, i.e. $m \cdot
(z_1,z_2)=(z_1+m,z_2+m')$. A choice of basis $\{\mu_1,\mu_2\}$ of $M$
determines an isomorphism
\[
\phi_{\mu_1,\mu_2}\colon M \backslash \CC^2 \rightarrow \CC^\times \times
\CC^\times, \qquad (z_1,z_2) \mapsto (u,v),
\]
where $\exp(2 \pi i z_1) = u ^{\mu_1}v^{\mu_2}$ and $\exp(2 \pi i z_2) = u
^{\mu_1'}v^{\mu_2'}$. 
A different choice of a basis is given by a matrix
$\left(\begin{smallmatrix}a&b\\c&d \end{smallmatrix}\right) \in
\GL_2(\ZZ)$ which induces the biholomorphic map $\psi \colon
\CC^{\times}\times \CC^{\times}\to \CC^{\times}\times \CC^{\times}$
given by
\[
(u,v) \mapsto (u^av^b,u^cv^d).
\]

We can always choose a basis of $M$ formed by totally
positive elements $\mu _{1},\mu _{2}\gg 0$. In this case, if
$\img(z_{1})$ and $\img(z_{2})$ tends to infinity (that is $(z_1,z_2)$
approaches the infinity cusp)
when at least one of $u$ or $v$ approaches $0$. Thus it is natural to
consider the
embedding $\CC^\times
\times \CC^\times \subset \CC^2$.

The map $\psi $ can be extended to an open subset of $\CC^{2}$ such
that its graph inside $\CC^2 \times 
\CC^2$ is closed. Therefore if we use it to glue together two copies of
$\CC^2$ we obtain a Hausdorff space. 

Let $M_+$ denote the elements of $M$ which are totally positive, and
consider the embedding of $M_+$ in $(\RR_+)^2$, given by
\[
m \mapsto (m,m').
\]
Denote by $A_j=(A_j^1,A_j^2)$, $j \in \Z$ the vertices of the boundary
of the convex hull of the image of $M_+$, ordered with the condition
$A_{j+1}^1 < A_j^1$ for all $j$. Any pair $(A_{j-1},A_j)$ is a basis
for $M$ as $\ZZ$-module (see \cite{VDG} Lemma 2.1).
In Appendix A we
describe the algorithm to compute such bases. 

Let $\sigma_j$ denote the cone spanned by $A_{j-1}$ and $A_j$, i.e.
\[
\sigma_j=\{s A_{j-1} + tA_j \; : \; s,t \in \R_+\}.
\]
We obtain a partial compactification of $M\backslash \CC^{2}$ by taking a 
copy of $\CC^2$ for each element $\sigma_j$ and gluing them together
in terms of the change of basis matrix (see \cite{VDG} page 31). By
the above comment we obtain a Hausdorff space. Hence we obtain a
partial compactification of $M\backslash \Siegel^{2}$ denoted
$Y^{+}$. 
Then $Y^{+}=M\backslash \Siegel^{2}\cup \bigcup_{j\in \ZZ}S'_{\infty,j}$, where each
$S'_{\infty,j}$ is a rational curve.
The space $Y^{+}$
is a Hausdorff space. The group of units $U_{\Om_K,\id{c}}^2$ acts
freely and properly discontinuously on $Y^{+}$ (\cite{VDG} Lemma
3.1 page 34). A local description of the desingularization of the
infinity cusp is obtained by taking the quotient of $Y^{+}$ by
$U_{\Om_K,\id{c}}^2$. Let $S_{\infty}$ denote the resolution divisor
of the infinite cusp and let $\{S_{\infty,j}\}_{j}$ be its irreducible
components. Then there is a one to one correspondence between the set
of classes of vertices $A_{j}$ under the action of
$U_{\Om_K,\id{c}}^2$ and the set of irreducible components of the resolution divisor
of the infinity cusp, and each irrecucible component is a rational curve.

Recall that we denote
$Z_\Gamma$ the desingularization obtained by applying this process to
each cusp of $X_\Gamma$.

If we apply the previous process to $M=\id{a}^{-1}$ and
$M=\id{a}^{-1}n$, where $n$ is a positive integer, since the two
lattices are homothetic, each choice of basis for $\id{a}^{-1}$ gives
a basis for $\id{a}^{-1}n$ and we get an holomorphic map between the
respective affine spaces given by sending $(u,v)$ to $(u^n,v^n)$. This
map is well behaved under gluing which gives a map
\[
\pi\colon Z_{\Gamma((n),\id{a})} \to Z_{\Gamma(\Om_K,\id{a})}.
\]
\begin{remark} \label{rem:ramification}
Let $E$ be a component of a cusp resolution of
$Z_{\Gamma((n),\id{a})}$ and $E'$ its image under $\pi$. It is clear
from this description  that the 
map between $E$ and $E'$ induced by $\pi $ has degree $n$ and $\pi$ is
ramified over 
$E'$ with ramification degree $n$ as well.
\end{remark}

\subsection{Algebraic Surfaces} 
Algebraic surfaces with vanishing irregularity are divided in four
types, one of them being of general type. For reasons that will become
clear later, it is this kind of surfaces the ones we need to work
with.

\begin{remark}\label{rem:ellipticpoints}
  If $\id{c}\subsetneq \Om_K$ is an integral ideal in $\Om_K$ with
  $\id{c}^2 \neq (2)$ and $\id{c}^2 \neq (3)$ then
  $X_{\Gamma(\id{c},\id{a})}$ has no elliptic points (see \cite{VDG}
  page 109).  In particular, in these cases, the surfaces
  $Z_{\Gamma(\id{c},\id{a})}$ and $Y_{\Gamma(\id{c},\id{a})}$ are the
  same.
\end{remark}

Recall the following classification.

\begin{thm}\label{thm:rationalsurfaces}
The Hilbert modular surface $Y_{\Gamma(\Om_K,\id{a})}$ is rational
for
\begin{itemize}
\item $D=5, 8, 12, 13, 17, 21, 24, 28, 33, 60$ if $\id{a}$ is in the
  principal genus.
\item $D=12$ if $\id{a}$ is not in the principal genus.
\end{itemize}
\end{thm}
\begin{proof}
This is Theorem $3.3$ of \cite{VDG}, Chapter VII p. 166.
\end{proof}
%\?\texttt{Tabla de racionales.}

\begin{thm}\label{thm:generaltype}
  The Hilbert modular surface $Y_{\Gamma(\id{c},\id{a})}$, with
  $\id{c} \neq \Om_K$ and $\id{a}$ in the genus $\gamma$ is of general
  type except in the following cases:
\begin{table}[h]
%\scalebox{0.9}{
\begin{tabular}{|c|c|c||c|c|c|}
\hline
$D$ & $\normid{c}$ &$\gamma$ & $D$ & $\normid{c}$ & $\gamma$\\
\hline
$5$ & $\{4,5\}$ &$+$& $8$ & $\{2,4\}$ & $+$\\
\hline
$12$ & $\{2,3,4,6\}$&$+,+$ & $12$ & $\{2,3\}$& $-,-$\\
\hline
$13$ & $\{3\}$ &$+$&$17$ & $\{2\}$ &$+$\\
\hline
$21$ & $\{3\}$&$-,-$&$24$ & $\{2\}$&$-,-$\\ 
\hline
$24$ & $\{3\}$&$+,+$& $28$ & $\{2\}$&$+,+$\\
\hline
$28$ & $\{3\}$&$-,-$&$33$ & $\{2\}$ &$-,-$\\
\hline
\end{tabular}
\end{table}

Furthermore, if $D>500$, then $Y_{\Gamma(\Om_K,\id{a})}$ is of general type as well.
\end{thm}

\begin{proof}
This is just part of Theorem 3.4 of \cite{VDG}, p. 167, where a general
classifications is given.
\end{proof}

Recall the following definition.

\begin{defi}
A smooth surface $S$ is called a \emph{minimal surface} if for any smooth
surface $S'$, any morphism  $S\to S'$ that is birational is an
isomorphism.
\end{defi}

From Castelnuovo's contractibility theorem, a minimal model of a
smooth surface can be obtained by contracting exceptional curves,
i.e. rational curves with self intersection number $-1$. We have the
following result.

\begin{prop}
Assume that $Y_{\Gamma(\id{c},\id{a})}$ is of general type and that
$Y_{\Gamma (\Om_{K},\id{a})}$ is not rational. If
$\normid{c} \ge C$, with 
\[
C = 3\left(\sum_{i=1}^h \sum_j (b_{i,j}-2)\right),
\]
where the first sum is over ideal class representatives $[\id{b}_{i}]$  of $\Om_{K}$ and
the $b_{i,j}$ are the self-intersection numbers of the components of
the cusp desingularization at $\id{b}_{i}$ (see Appendix A), then
$Y_{\Gamma(\id{c},\id{a})}$ is minimal.
\label{prop:Cconstant}
\end{prop}

\begin{proof}
The statement corresponds to the first case of Theorem 7.19 of \cite{VDG}, p. 184.
\end{proof}

\begin{remark}
This gives an effective bound for the level $\id{c}$ needed for the
Hecke/Sturm bounds (of sections \ref{sec-Hecke} and
\ref{sec-Sturm}). 
It will become clear that the smaller $\normid{c}$
we take, the better the bound gets, so we will say a few more words on
how to improve this norm.
\end{remark}

Recall the definition of the Hirzebruch-Zagier
cycles (which correspond to the modular curves inside the Hilbert
modular surfaces). A matrix $B$ in $M_2(K)$ is called skew-hermitian
if $B^t = -B'$, where the superscript $t$ means the transpose. Let
$\id{a} \in \Om_K$ be an ideal of norm $A$. A skew-hermitian form $B$
is called integral with respect to $\id{a}$ if it is of the form 
\[
B = \left( \begin{array}{cc}
a \sqrt{D} & \lambda \\
-\lambda' & \frac{b}{A}\sqrt{D} \end{array} \right),
\]
with $a,b \in \ZZ$ and $\lambda \in \id{a}^{-1}$. The integral form
$B$ is called primitive if it is not divisible by a natural number
greater than $1$, i.e. if $B$ is not of the form $m \tilde{B}$, with
$\tilde{B}$ integral with respect to $\id{a}$ and $m >1$. If we denote
by $\bfC(N)$ the set of skew-hermitian, integral with respect to
$\id{a}$, primitive matrices of determinant $N/A$, then the cycle $F_N$
is defined by
\begin{equation}
F_N = \bigcup_{B \in \bfC(N)} \left\{(z_1,z_2) \in \Siegel^2 \cup \PP^1(K): (z_2 \quad 1)B\left(\begin{array}{c} z_1\\ 1\end{array} \right)=0 \right\}.
\end{equation}

Abusing the notation, we will also denote by $F_{N}$ the divisor on any
modular surface obtained as the closure of the image of $F_{N}$. By
the context it will be clear in which surface we are considering them.

The following conjecture is stated as Conjecture (7.13) in \cite{VDG}.

\begin{conjecture}
If $Y_{\Gamma(\Om_K,\id{a})}$ is not rational, then the canonical divisor
can be written as a rational
positive linear combination of resolutions curves and the divisors
$F_N$.
\label{conj:canonicaldivisor}
\end{conjecture}

\begin{remark} \label{remark:conjecture}
When $\id{a}$ is in the genus of $\Om_K$ or
  $(\sqrt{D})$, 
this conjecture is known in the following cases
\begin{enumerate}
\item When $Y_{\Gamma(\Om_K,\id{a})}$ is not of general type.
\item \cite{Hermann,Hermann2} When $D\equiv 1 \pmod{8}$ and either
  \begin{enumerate}
  \item there is a divisor $a$ of $D$ with $a\not\equiv 1 \pmod{8}$;
  \item there are two integers $n,m>0$ with $m\equiv 7 \pmod{8}$ and
    $D=(m^{2}-8)/n^{2}$. 
  \end{enumerate}
\item \cite{Freitag} When $D\not \equiv 1 \pmod{8}$.
\end{enumerate}
\end{remark}

Assume now that $Y_{\Gamma (\id{c},\id{a})}$ is of general type and
$Y_{\Gamma (\Om_K,\id{a})}$ is not rational. We want to improve our
criterium for minimality of $Y_{\Gamma (\id{c},\id{a})}$ assuming that
Conjecture \ref{conj:canonicaldivisor} is true for $Y_{\Gamma
  (\Om_K,\id{a})}$. By Proposition 7.18 of \cite{VDG} (p. 183), if $E$ is an
exceptional curve in $Y_{\Gamma (\id{c},\id{a})}$, then its image in
$Y_{\Gamma (\Om_K,\id{a})}$ is also exceptional. If Conjecture
\ref{conj:canonicaldivisor} is true for $Y_{\Gamma (\id{c},\id{a})}$,
the exceptional curves in this surface are components of the divisors
$F_{N}$. Therefore any exceptional curve in $Y_{\Gamma
  (\id{c},\id{a})}$ is also a component of a divisor $F_N$.
If for example $6 \mid \id{c}$, then the components of the
curves $F_{N}$ have genus greater than $1$ (see for example
\cite{Shimura}, formula (1.6.4), page 23) and are therefore not
exceptional, hence the surface
$Y_{\Gamma(\id{c},\id{a})}$ is minimal for this level. Actually we can
do a little better.

\begin{thm}
Assume that $Y_{\Gamma(\Om_K,\id{a})}$ is not rational and
Conjecture~\ref{conj:canonicaldivisor} is true for
this surface. If
$n \ge 3$ is an integer and $Y_{\Gamma ((n),\id{a})}$ is of general type,
then $Y_{\Gamma((n),\id{a})}$ is minimal.
\label{thm:bestbound}
\end{thm}

\begin{proof}  
Recall that $Z_{\Gamma (\Om_{K},\id{a})}$ is the resolution of the cusps of
$X_{\Gamma (\Om_{K},\id{a})}$ but without resolving the elliptic
points which is a $\Q$-variety. By Remark \ref{rem:ellipticpoints}, 
$Y_{\Gamma((n),\id{a})}$ agrees with $Z_{\Gamma((n),\id{a})}$ and
hence we get the following diagram
\begin{displaymath}
  \xymatrix{  &
    Y_{\Gamma ((n),\id{a})}\ar[d]^{\pi }\\
    Y_{\Gamma (\Om_K,\id{a})}\ar[r]^{f} & 
    Z_{\Gamma (\Om_K,\id{a})},
  }
\end{displaymath}
where $f$ is the resolution at the elliptic points.

We need to show that there are no exceptional curves on $Y_{\Gamma
  ((n),\id{a})}$. Assume that there is such an exceptional curve
$A$. Let $C'$ be its image in $Z_{\Gamma (\Om_K,\id{a})}$ and $C$ the
strict transform of $C'$ in $Y_{\Gamma (\Om_K,\id{a})}$. As we
mentioned previously, by Proposition 7.18 of \cite{VDG}, the curve
$C$ is exceptional. 
By Theorem 7.11 of \cite{VDG} (p. 181), $C$ (hence $A$) is a
component of a divisor 
$F_{N}$ for $N=1, 2, 3$ or $4$ (and $9$ if $3 \mid
D$).

We will show that $A\cdot A<-1$ contradicting the
assumption. To this end  
we start by computing the self-intersection of $C'$. 
We have the relation
\[
C' \cdot C' = f^*(C') \cdot f^*(C').
\]
Using the desingularization of the components of $F_i$ given in \cite{VDG} (page
169), we get the following cases:

\noindent $\bullet$ {\bf The case $i=1$:} The curve $C'$ goes through
an elliptic point of order $2$ and an elliptic point of order
$3$. While computing the desingularization at the order $2$ point, we
get a $\PP^1$ with self-intersection $-2$ and while computing the
desingularization of elliptic point of the order $3$ we get a $\PP^1$
with self-intersection $-3$ (see Figure (2) in \cite{VDG}, page
169). Let $E_2$ and $E_3$ be these two exceptional divisors. We can
write $f^*(C') = C+aE_2+bE_3.$ Since $f^*(C') \cdot E_2 = f^*(C') \cdot
E_3 = 0$, we get
  \[f^*(C') = C+ \frac{1}{2}E_2+\frac{1}{3}E_3.\]
  Therefore
\[
f^*(C') \cdot f^*(C') = C\cdot C + C\cdot E_2 +\frac{2}{3} C\cdot E_3 + \frac{1}{4}E_2\cdot E_2 + \frac{1}{9}E_3 \cdot E_3 = -\frac{1}{6}.
\]
\noindent $\bullet$ {\bf The case $i=2$:} The curve $C'$ goes through
an elliptic 
  point of order $2$. While computing the desingularization at the
  order $2$ point, we get a $\PP^1$ with self-intersection $-2$ (see
  Figure (3) in \cite{VDG}, page 169). Let $E_2$ be the exceptional
  divisor, so  $f^*(C') = C+aE_2$. Since $f^*(C') \cdot E_2 =
  0$, we get that $a= \frac{1}{2}$, and
\[
f^*(C') \cdot f^*(C') = C\cdot C + C\cdot E_2 + \frac{1}{4}E_2\cdot E_2 = -\frac{1}{2}.
\]
\noindent $\bullet$ {\bf The case $i=3$:} Since
$Y_{\Gamma(\Om_K,\id{a})}$ is not rational, $D \neq 12$. Then the curve
  $C'$ goes through an elliptic point of order $3$. While computing
  the desingularization at the order $3$ point, we get a $\PP^1$ with
  self-intersection $-3$. Let $E_3$ be the exceptional divisor, then
  $f^*(C') = C+bE_3$. Since $f^*(C') \cdot E_3 = 0$, we get that
  $b= \frac{1}{3}$, and
\[
f^*(C') \cdot f^*(C') = C\cdot C +\frac{2}{3} C\cdot E_3 + \frac{1}{9}E_3\cdot E_3 = -\frac{2}{3}.
\]
\noindent $\bullet$ {\bf The case $i=4$:} Since
$Y_{\Gamma(\Om_K,\id{a})}$ is not rational, $D \neq 8$. If $2 \mid D$,
  then the situation is the same as the case $i=2$. If $D \equiv 1
  \pmod 8$ then $C'$ does not go through any elliptic point, hence
  the self intersection is $-1$.  If $D \equiv 5 \pmod 8$ then the
  curve $C'$ goes through two elliptic points of order $3$. While
  computing the desingularization at the two order $3$ points, we get
  two copies of $\PP^1$ with self-intersection $-3$. Let $E_3$ and
  $E_3'$ be the exceptional divisors. Then $f^*(C') =
  C+\frac{1}{3}E_3 + \frac{1}{3}E_3'$, and
\[
f^*(C') \cdot f^*(C') = C\cdot C + \frac{2}{3}C\cdot (E_3+E_3') + \frac{1}{9}(E_3\cdot E_3+E_3'\cdot E_3') = -\frac{1}{3}.
\]
\noindent $\bullet$ {\bf The case $i=9$:} Again we use $D \neq 12$. If $3 \nmid D$,
  then the curve $C'$ does not go through any elliptic point. If $3
  \mid D$, and $D \neq 105$, then $C'$ goes through an elliptic point
  of order $3$, so the blow up gives a $\PP^1$ with self intersection
  number $-3$ (see the first Figure of \cite{VDG} page 170), so we are
  in the same situation as the case $i=3$.

If $D=105$, the picture is similar, but in this case some
components are not disjoint any more. Even though, the same computation applies.

\medskip

Let $g$ denote the
degree of $\pi$ and $d$ the degree of the morphism induced by
$\pi$ between the modular curve $A$ and its image $C'$. Since the
morphism $\pi$ is not ramified over $C'$, the preimage
of $C'$ consists on $c=g/d$ curves which are translates of $A$,
\begin{displaymath}
  \pi ^{\ast}(C')\cdot \pi ^{*}(C')=g C'\cdot C'\qquad\text{ and }\qquad\pi
  ^{\ast}(C')\cdot \pi ^{\ast}(C')\ge c A\cdot A.
\end{displaymath}
Therefore 
\[
 A\cdot A \le d C' \cdot C'.
\]
Note that $d = [PSL_2(\ZZ):\Gamma(n)]$, where $\Gamma(n)$ is the
classical congruence subgroup. Since $n\ge 3$, $d >6$ and $A \cdot
A<-1$. Thus $A$ is not exceptional. 
\end{proof}

\begin{remark}
  It is clear that if $Y_{\Gamma((n),\id{a})}$ is a minimal surface of
    general type and $m$ is a positive integer, then
    $Y_{\Gamma((mn),\id{a})}$ is also a minimal surface of general
      type.
\end{remark}

\begin{summary}\label{summ:1}
In this section we have obtained the following results:
\begin{itemize}
\item If $D=5, 8, 12, 13, 17, 21, 24, 28, 33, 60$ and $\id{a}$ is in
  the principal genus or $D=12$ and $\id{a}$ is not in the principal
  genus, then $Y_{\Gamma(\Om_K,\id{a})}$ is rational so the
  previous results do not apply. This case will be treated separately
  in Appendix~\ref{appendix:b}.
\item If $Y_{\Gamma(\Om_K,\id{a})}$ is not rational and 
\[
n\ge\sqrt{3\left(\sum_{i=1}^h \sum_j (b_{i,j}-2)\right)},
\]
then $Y_{\Gamma((n),\id{a})}$ is a minimal surface of general type.
\item If $Y_{\Gamma(\Om_K,\id{a})}$ is not rational and satisfies
  Conjecture~\ref{conj:canonicaldivisor} (see
  Remark~\ref{remark:conjecture}) then $Y_{\Gamma((n),\id{a})}$ is a
  minimal surface of general type for $n\ge 3$.
\end{itemize}
\end{summary}

\section{Hilbert modular forms}

In this section we recall the definition and basic properties of
Hilbert modular forms.

\begin{defi}
  Let $\Gamma_{\id{a}}$ be a congruence subgroup, and $k_{1}$ and
  $k_{2}$ be integers such that $k_1 \equiv k_2 \pmod{2}$. A
  holomorphic function $G\colon \Siegel^{2}\to \CC$ is called a
  \emph{Hilbert modular form} of weight $\bfk=(k_{1},k_{2})$ for the
  group $\Gamma_{\id{a}}$ if for all $\gamma =
  \begin{pmatrix}
    a & b\\ c&d
  \end{pmatrix}\in \Gamma_{\id{a}}$
  one has, for each $\bfz=(z_{1},z_{2})\in \Siegel^{2}$,
  \begin{equation}
    \label{eq:2}
    G(\gamma \bfz)=(cz_{1}+d)^{k_{1}}(c'z_{2}+d')^{k_{2}}G(\bfz).
  \end{equation}
  If $k$ is an integer and 
  $G$ is a modular form of weight $\bfk=(2k,2k)$, we will call it a
  modular form of
  parallel weight $2k$. We will denote by $M_{\bfk}(\Gamma_{\id{a}} )$ the space
  of all modular forms of weight $\bfk$ and by $M_{2k}(\Gamma_{\id{a}} )$ the
  space of all modular forms of parallel weight $2k$.
\end{defi}

Let $G$ be a Hilbert modular form of weight $(k_{1},k_{2})$. It admits
a Fourier expansion in each cusp. Since all the cusps are conjugate to
the infinity cusp (possibly altering the ideals) by an element of
$\PSL_2(K)$, we will just recall the case of the infinity cusp. Since
$\Gamma_{\id{a}}$ is a congruence group, the isotropy group of the
cusp $(1:0)$ contains some $G(M,V)$ (since for example for
$\Gamma(\id{c},\id{a})$ it equals
$G(\id{a}^{-1}\id{c},U_{\Om_{K},\id{c}}^{2})$). The modularity
condition implies that, if $m\in M$ and $\epsilon \in
U_{\Om_{K},\id{c}}$ then
\begin{align}
  G(z_{1}+m,z_{2}+m')&=G(z_{1},z_{2}),\label{eq:3}\\
  G(\epsilon ^{2}z_{1},{\epsilon '}^{2}z_{2})&=\epsilon
  ^{-k_{1}}{\epsilon'} ^{-k_{2}}G(z_{1},z_{2}).\label{eq:4} 
\end{align}
The periodicity condition \eqref{eq:3} implies that $G$ admits the
Fourier expansion
\[
G= \sum_{\substack{\xi \in M^{\vee}}} a_\xi \exp(2\pi i(\xi
z_1 + \xi' z_2)), 
\]
where $M^{\vee}$ is the set of $\xi\in K$ such that $\Tr(m\xi)\in \ZZ$
for all $m\in M$. Let $M^{\vee}_{+}$ denote the set of totally
positive elements of $M^{\vee}$. Then the holomorphicity of $G$
implies that the only non-zero coefficients $a_\xi$ of the above
expansion are $a_0$ and $a_\xi$ with $\xi\in M^{\vee}_{+}$. Hence
\[
G= \sum_{\xi \in M^{\vee}_{+}\cup\{0\}} a_\xi \exp(2\pi i(\xi
z_1 + \xi' z_2)).
\]

The modularity equation
\eqref{eq:4} implies that the
coefficients of the Fourier expansion satisfy the condition
\begin{equation}
  \label{eq:12}
 a_{\xi\epsilon ^{2}}=\epsilon ^{k_{1}}\epsilon
 '{}^{k_{2}}a_{\xi}\quad\text{ for all }\quad\epsilon \in
 U_{\Om_K,\id{c}}. 
\end{equation}
In particular, if $G$ is of parallel weigh $2k$ then $a_{\xi\epsilon
  ^{2}}=a_{\xi}$. 

By means of the Fourier expansion, we see that every modular form
determines a holomorphic function in an analytic neighborhood of each
cusp. 

\begin{defi}
  \begin{enumerate}
  \item A Hilbert modular form $G$ is called a cusp form if, for each
    cusp, the  coefficient $a_{0}$ of the Fourier expansion of $G$ is
    zero. We denote by $S_{\bfk}(\Gamma ) \subset M_{\bfk}(\Gamma ) $
    the space of modular cusp forms of weight $\bfk$ and by
    $S_{2k}(\Gamma ) \subset M_{2k}(\Gamma ) $ 
    the space of modular cusp forms of parallel weight $2k$.
  \item Let $G$ be a Hilbert modular form of parallel weight $2k$ for
    the group $\Gamma (\id{c},\id{a})$ and $c_{i}$ a cusp of $X_{\Gamma
  (\id{c},\id{a})}$. Let $S_{i}$ be the resolution divisor of
$c_{i}$ in $Y_{\Gamma (\id{c},\id{a})}$. The modular form $G$
determines a holomorphic function $f$ in an analytic neighborhood
$U_{i}$ of
$S_{i}$. We say that $G$
vanishes with order $a$ at the cusp $c_i$ if, the divisor
$\Div(f) -aS_i$ 
is effective in $U_{i}$. We will write $\ord_{c_{i}}G=a$ if $G$
vanishes at the cusp $c_{i}$ with order $a$ but does not vanish with
order $a+1$.  
  \end{enumerate}
\end{defi}

The vanishing of a Hilbert modular form at a cusp can be read
from the Fourier expansion. For simplicity we will treat only the case
of the infinity cusp. Let
  $\{A_{j}\}_{j\in J}$ be a set of representatives under the action of
  $V$, of the corners of the
  convex hull of 
  $M_{+}$ (see Section \ref{sec:cusp-resolution}).
\begin{lemma}\label{lemm:1}
  Let $G$ be a modular form of parallel weight $2k$  for a congruence
  subgroup 
$\Gamma _{\id{a}}$ and
\begin{displaymath}
  G= \sum_{\xi \in M^{\vee}_{+}\cup\{0\}} a_\xi \exp(2\pi i(\xi
z_1 + \xi' z_2)), 
\end{displaymath}
its Fourier expansion at the infinity cusp. Then
\begin{displaymath}
  \ord_{c_{1}}G=\inf\{\Tr(\xi A_{j})\mid j\in J, a_{\xi}\not =0\}.
\end{displaymath}
Thus, $G$  vanishes with
  order $a$ at the infinity cusp if and only if  $a_{\xi}=0$ for all
  $\xi\in M^{\vee}_{+}\cup\{0\}$ such that there is a $j\in J$ with
  $\Tr(\xi A_{j})< a$. 
\end{lemma}
\begin{proof}
  By \eqref{eq:12}, the vanishing condition for the coefficients of
  the Fourier expansion is equivalent to the condition $a_{\xi}=0$ for all
  $\xi\in M^{\vee}_{+}\cup\{0\}$ such that there is a $j\in \ZZ$ with
  $\Tr(\xi A_{j})< a$.
  Let $A_{j},A_{j+1}$ be a totally 
  positive basis of $M$ as in Section \ref{sec:cusp-resolution}. To
  this basis, 
  there is associated a
  local analytic chart of a piece of the cusp resolution. Let $u,v$ be 
  the local coordinates of this chart. The divisors $u=0$ and $v=0$
  correspond to components of the cusp resolution divisor.  
  With these coordinates, the
  Fourier expansion of $G$, is given by
  \begin{displaymath}
    G(u,v)=\sum_{\xi \in M^{\vee}_{+}\cup\{0\}} a_\xi
    u^{\Tr(\xi A_{j})} v^{\Tr(\xi A_{j+1})}.
  \end{displaymath}
  Thus, the lemma follows directly from the definition of order of
  vanishing at 
  a cusp.
\end{proof}

From the lemma, it is clear that a modular form is a cusp form
if and only if it vanishes at each cusp with order one.

Let $G\in M_{2}(\Gamma )$ be a modular form of parallel weight
$2$. Then $\omega _{G}=Gdz_{1}\land 
dz_{2}$ is a $\Gamma $-invariant differential form on
$\Siegel^{2}$. Thus, it defines a differential form on  $\Gamma \backslash
\Siegel^{2}$, hence on an open subset of $Y_{\Gamma }$. It can be seen
(\cite[Ch 3. \S 3]{VDG})
that $\omega _{G}$ can be extended to a differential form on
$Y_{\Gamma }$ that is regular on the resolution divisors of the elliptic fixed
points and has at most logarithmic poles at the resolution divisors of
the cusps. This gives us the identifications
\begin{displaymath}
  S_{2}(\Gamma )\overset{\simeq}{\longrightarrow}H^{0}(Y_{\Gamma },\Om(K_{Y_\Gamma })),\qquad
M_{2}(\Gamma )\overset{\simeq}{\longrightarrow}H^{0}(Y_{\Gamma },\Om(K_{Y_{\Gamma }}+S)),
\end{displaymath}
where $K_{Y_{\Gamma }}$ is the canonical divisor of $Y_{\Gamma }$ and
$S=\sum S_{i}$ is the sum of the resolution divisors of all the cusps.

From the above identifications one can derive the following result.

\begin{prop}\label{prop:1}
  Let $\Gamma $ be a congruence subgroup, $\{c_{1},\dots,c_{h}\}$ the
  set of cusps of $X_{\Gamma }$, $S_{i}$ the resolution divisor of
  $c_{i}$ on $Y_{\Gamma }$ and  $S=\sum S_{i}$. Fix integers $1\le
  i_{0}\le h$, $a,s\ge 0$. Then we can identify the space of all
  modular forms for $\Gamma $ of parallel weight $2k$, vanishing order
  at least $s$ at all the cusps and at least $a+s$ at the cusp $i_{0}$
  with the space of global sections 
  $H^{0}(Y_{\Gamma },\Om(kK_{Y_{\Gamma }}+(k-s)S-aS_{i_{0}}))$.
\end{prop}

\section{Hecke Bound}
\label{sec-Hecke}

In this section we will derive a Hecke type bound for Hilbert modular
forms for the group $\Gamma (\Om_{K},\id{a})$.  We will assume that
$D>0$ is such that $Z:=Z_{\Gamma (\Om_K,\id{a})}$ is not
rational. Choose $n$ such that $Z_n:=Z_{\Gamma ((n),\id{a})}$ is a
minimal surface of general type (see Summary~\ref{summ:1}).

Let $S$ be the cusp resolution divisor on $Z$. We order the cusps of
$X_{\Gamma (\Om_K,\id{a})}$ as $c_{i}$, $i=1,\dots,h$ and we decompose $S$ as
\begin{displaymath}
  S=\sum_{i=1}^{h}S_{i},
\end{displaymath}
where $S_{i}$ is the resolution divisor over the cusp $c_{i}$.
 
For each $i=1,\dots,h$, let $b_{i,j}\ge 2$ be the integers that appear
in the cusp 
desingularization process of $X_{\Gamma (\Om_K,\id{a})}$ as explained
in Appendix~\ref{appendix:a}. Let $1\le i_{0}\le h$.

\begin{thm}[Hecke bound] \label{thm:1}
With the previous hypothesis on $D$ and $n$,
let $G$ be a Hilbert modular form of parallel weight $2k$ for
$\Gamma(\Om_{K},\id{a})$ and suppose that $\ord_{c_{i}}G\ge s$ for
$i=1,\dots,h$ and $\ord_{c_{i_{0}}}G\ge a+s$, with
\[
a > \frac{4kn \zeta_K(-1)}{\sum_j (b_{i_{0},j}-2)}-s
\left(\frac{\sum_{i=1}^h \sum_j
    (b_{i,j}-2)}{\sum_j(b_{i_{0},j}-2)}\right).
\]
Then $G$ is zero.
\end{thm}
\begin{remark}\label{rem:1}
  For each $i$ there is a $j$ with $b_{i,j}>2$, since otherwise, the
  desingularization divisor of the cusp $c_{i}$ has self-intersection
  $0$, which contradicts Hodge index theorem. Therefore, the
  denominators in the above expression are different from zero.
\end{remark}

Before proving the theorem we need some known results.
Recall that $Z_n$ does not have elliptic points, hence is already
smooth. Let $\pi \colon Z_n\to Z$ be the projection and $d$ its
degree. Let $c'$ be the number of cusps of
$X_{\Gamma ((n),\id{a})}$ that are over a cusp of $X_{\Gamma (\Om_K,\id{a})}$.
By \cite[Lemma 5.2, Chapter
IV]{VDG} and its proof
\begin{equation}
  \label{eq:10}
  d=n^{2}c'[U^{2}_{\Om_{K}}:U^{2}_{\Om_{K},(n)}].
\end{equation}
For each $i'=1,\cdots,hc'$, let
$b'_{i',j}$ be the integers that appear
in the cusp 
desingularization process of $X_{\Gamma ((n),\id{a})}$. Let $c_{i'}$
be a cusp of  $X_{\Gamma ((n),\id{a})}$ over a cusp $c_{i}$ of
$X_{\Gamma (\Om_K,\id{a})}$. Then the sequence $(b'_{i',j})_{j}$ is a
repetition of $[U^{2}_{\Om_{K}}:U^{2}_{\Om_{K},(n)}]$ times the
sequence $(b_{i,j})_{j}$. Therefore
\begin{equation}
  \label{eq:11}
  \sum_{i'=1}^{c'h}\sum
  _{j}(2-b'_{i',j})=\frac{d}{n^{2}}\sum_{i=1}^{h}\sum
  _{j}(2-b_{i,j}).
\end{equation}

Let $S'$ be the 
cusp resolution divisor of $Z_n$ and, for $i=1,\dots h$, let $S_{i}'$
be the sum of the resolution divisor of all the cusps on $Z_n$ over
$c_i$. By Remark \ref{rem:ramification} we have
\begin{equation}\label{eq:1}
  \pi ^{\ast}(S)=nS'\qquad\text{ and }\qquad \pi ^{\ast}(S_{i})=nS_{i}'.
\end{equation}
By the geometry of the cusp resolutions (see Appendix
\ref{appendix:a}) we have
\begin{displaymath}
  S_{i}\cdot S_{l}= 
  \begin{cases}
    \sum
  _{j}(2-b_{i,j})&\text{ if }i=l,\\
  0&\text{ if }i\not=l.
  \end{cases}
\end{displaymath}
From this, using equation \eqref{eq:1}, we deduce
\begin{equation}
  \label{eq:9}
  S'_{i}\cdot S'_{l}=  
  \begin{cases}
    \frac{d}{n^{2}}\sum
  _{j}(2-b_{i,j})&\text{ if }i=l,\\
  0&\text{ if }i\not=l.
  \end{cases}  
\end{equation}

Let $K_{Z_{n}}$ be the canonical divisor of $Z_{n}$. Since each divisor $S_{i}$
is a cycle of rational curves, it has arithmetic genus $1$. Then
the adjunction formula implies that
\begin{equation}
  \label{eq:5}
  (K_{Z_{n}}+S'_{i})\cdot S'_{i}=0, \quad (K_{Z_{n}}+S')\cdot S'=0. 
\end{equation}
Therefore
\begin{equation}
  \label{eq:8}
  K_{Z_{n}}\cdot S'_{i}=\frac{d}{n^{2}}\sum
  _{j}(b_{i,j}-2),\qquad  K_{Z_{n}}\cdot S'=\frac{d}{n^{2}}\sum_{i=1}^{h}\sum
  _{j}(b_{i,j}-2).
\end{equation}

Moreover, by  \cite[Chapter IV, Theorem 2.5]{VDG} (page 64), 
\cite[Chapter IV, Theorem 1.1]{VDG}, (pp. 59) and equation \eqref{eq:11},
\begin{equation}\label{eq:6}
  K_{Z_{n}}\cdot K_{Z_{n}}=2\Vol(Z_{n})+\frac{d}{n^{2}}\sum_{i=1}^{h}\sum
  _{j}(2-b_{i,j})=4 d\zeta_K(-1) +\frac{d}{n^{2}}\sum_{i=1}^{h}\sum
  _{j}(2-b_{i,j}).
\end{equation}
\begin{proof}[Proof of Theorem \ref{thm:1}]
Since $G$ is a Hilbert
modular form of parallel weight $2k$ that vanishes with order $s$ at
every cusp and with order $a+s$ at
the cusp $c_{i_{0}}$, by Proposition \ref{prop:1}, it determines a
global section of 
$\Om(k(K_{Z_{n}}+S')-snS'-anS_{i_0}')$. Since $Z_n$ is a minimal surface
of general type, $K_{Z_{n}}$ is NEF. Hence, if $G\not=0$,
the intersection number $K_{Z_{n}} \cdot (k(K_{Z_{n}}+S')-snS'-anS_{i_0}')$
must be non-negative. If we prove that this
number is negative, we are done. Using equations
\eqref{eq:8} and \eqref{eq:6},
 we obtain
\begin{multline}\label{eq:7}
K_{Z_{n}} \cdot (k(K_{Z_{n}}+S')-snS'-anS_{i_0}')\\
= d\left(4k\zeta_k(-1) +\frac{s}{n} \sum_{i=1}^h \sum_j (2-b_{i,j})
  +\frac{a}{n} \sum_j(2-b_{i_0,j})\right)
\end{multline}
proving the Theorem.
\end{proof}

By virtue of Lemma~\ref{lemm:1}, we can state the same result in terms
of  Fourier expansions. For simplicity we will treat only the case of
the infinity cusp. Assume that we have numbered the cusps in such a
way that the infinity cusp is $c_{1}$. The lattice corresponding to
the isotropy group of the infinity cusp is $M=\id{a}^{-1}$ and the
group of units $V$ equals $U^{2}_{\Om_K}$. Let $\{A_{j}\}_{j\in J}$ be
a set of representatives under the action of $U^{2}_{\Om_K}$ of the corners of 
the convex hull of $(\id{a}^{-1})_{+}$.

\begin{coro} \label{coro:mainHecke}
  With the same hypothesis on $D$ and $n$, let $G$ be a Hilbert
  modular form of parallel weight $2k$ for $\Gamma(\Om_{K},\id{a})$ which
  vanishes with order $s$ at all the cusps. Let $a$ be an integer with
  \[
  a > \frac{4kn \zeta_K(-1)}{\sum_j (b_{1,j}-2)}-s
  \left(\frac{\sum_{i=1}^h \sum_j
      (b_{i,j}-2)}{\sum_j(b_{1,j}-2)}\right).
  \] 
 Suppose that the Fourier expansion of $G$ at the infinity cusp is 
  \[
  G= \sum_{\xi \in (\id{a}^{-1})_{+}^{\vee}\cup\{0\}} a_\xi \exp(2\pi i(\xi
  z_1 + \xi' z_2)).
  \]
If $a_{\xi}=0$ for all $\xi\in (\id{a}^{-1})_{+}^{\vee}\cup\{0\}$ such
that there is a $j\in J$ with $\Tr(\xi A_{j})< a+s$, then $G=0$.
\end{coro}

\begin{remark}
  \begin{enumerate}
  \item   Although both Theorem~\ref{thm:1} and Corollary~\ref{coro:mainHecke}
  are stated for forms vanishing with order $s$ at all cusps, the two
  usual cases are $s=0$ for a general Hilbert
  modular form and $s=1$ for a cusp form.
\item It is clear from Theorem~\ref{thm:1} and
  Corollary~\ref{coro:mainHecke} that the smaller the $n$, the better
  the bound we get.
  \end{enumerate}
\end{remark}

\begin{remark}
  The bound we got in Theorem~\ref{thm:1} relies on the choice of an
  auxiliary positive integer $n$ such that $Z_{n}$ is a minimal
  surface of general type, and there is a dependence of $n$ in the
  formula. We can think of this dependence in a somehow different
  way. We need to construct a NEF divisor in some surface. What we did
  was to start with a parallel weight $2k$ Hilbert modular form $G$
  for $\Gamma(\Om_{K},\id{a})$ and considered its pullback to $Z_{n}$,
  where we can identify a NEF divisor, namely the canonical
  divisor. But we can do the opposite, recall the following result
  concerning NEF divisors under maps.

\medskip

\noindent{\bf Fact:} Let $\pi : X \to Y$ be a surjective generically finite map
    between surfaces. Let $\bfD \subset Y$ be a Cartier divisor. Then
    $\bfD$ is a NEF divisor if and only if $\pi^*(\bfD)$ is a NEF
    divisor.

\medskip

This implies that we can do the computations in ``level 1''. Take any
(rational) divisor $\bfD$ in $Z_1$ whose pullback to $Z_n$ is the
canonical divisor and compute the intersection numbers with it (which
of course gives the same bound). Thus, the dependence on $n$ does not
come from where we compute the intersection numbers but from where we
can identify a NEF divisor.

  Thus, there are two ways for getting a better
  bound in some particular cases:
  \begin{enumerate}
  \item If one can compute the cone of NEF divisors, one can make the
    same computations for each generator of the NEF cone to get the
    best bound. 
  \item If $Y_{\Gamma(\Om_K,\id{a})}$ is of general type (which
    happens for example if $D>500$), one can compute its minimal
    model, and take as NEF divisor any divisor $\bfD$ in
    $Z_{\Gamma(\Om_K,\id{a})}$ whose pullback to the minimal model is
    the canonical divisor 
    to get a bound with ``$n=1$''.
  \end{enumerate}
\end{remark}

\section{Sturm bound}
\label{sec-Sturm}

To make the computation of the previous section work over a finite
field, we need to use the integral structure of the Hilbert modular
surface. Such structure comes from
their moduli interpretation and has been developed in \cite{Rapoport},
\cite{Chai} and \cite{Pappas}, see also the book \cite{Go}. 

Let $D>0$ be a fundamental discriminant.
Let $\id{a}$ be a fractional ideal, $n \ge 3$ a positive integer and
$\zeta_n$ a primitive $n$-th root 
of unity. Consider the modular surface $Y_{\Gamma ((n),\id{a})}$ and
let $S'$ be the cusp resolution. The first input we need is the
existence of a nice regular model of $Y_{\Gamma ((n),\id{a})}$.

\begin{thm}\label{thm:3}  There exist a regular scheme $\specY_{\Gamma
    ((n),\id{a})}$, 
  smooth, proper and flat over $\ZZ[1/(Dn),\zeta_n]$, such that
  \begin{displaymath}
    \specY_{\Gamma
      ((n),\id{a})}\underset{\ZZ[1/(Dn),\zeta_n]}{\times}\Spec(\CC)=Y_{\Gamma
      ((n),\id{a})}.
   \end{displaymath}
   Moreover, there is a relative normal crossing divisor $\Sc'$ of
   $\specY_{\Gamma ((n),\id{a})}$ whose restriction to $Y_{\Gamma
     ((n),\id{a})}$ is $S'$.
\end{thm}
\begin{proof}
See \cite{Chai} Theorem 3.6, \cite{Rapoport} Th\'eor\`eme
5.1 and Corollaire 5.3. and \cite{Pappas} Theorem 2.1.2.
\end{proof}

The second input we need is the $q$-expansion principle. Let $K$ be
the canonical divisor of $Y_{\Gamma ((n),\id{a})}$ and let $\bfK$ be
the relative canonical divisor of $\specY_{\Gamma ((n),\id{a})}$. Let $R$ be a
subalgebra of $\CC$ that contains $\ZZ[1/(Dn),\zeta_{n}]$. We will denote
by $\specY_{\Gamma ((n),\id{a}),R}$, $\bfK_{R}$ and $\Sc'_{R}$ the objects obtained
after extending scalars to $R$. We know that a modular form of parallel 
weight $2k$ determines a section of $\Om_{Y_{\Gamma ((n),\id{a})}}(k(K+S'))$.

\begin{thm} \label{thm:5} Let $G$ be a Hilbert modular form of
  parallel weight $2k$ 
  for $\Gamma ((n),\id{a})$, and  let 
\[
  G= \sum_{\xi \in M^{\vee}_+\cup\{0\}} a_\xi \exp(2\pi i(\xi
  z_1 + \xi' z_2)),
  \]
be its Fourier expansion at a cusp. Then the form $G$ determines a section of
$\Om_{\specY_{\Gamma ((n),\id{a})},R}(k(\bfK_{R}+\Sc'_{R}))$ if and
only if $a_{\xi}\in R$ for all $\xi\in M^{\vee}_+\cup\{0\}$.
\end{thm}
\begin{proof}
  See \cite{Chai} Theorem 4.3 and \cite{Rapoport} Th\'eor\`eme 6.7.
\end{proof}

Finally we need to know that the fibers of $\specY_{\Gamma
  ((n),\id{a})}$ are also minimal surfaces. 

\begin{prop}\label{prop:3} Let $\Om$ be a Dedekind domain contained in
  $\CC$ that 
  contains $\ZZ[1/(Dn),\zeta_{n}]$. Let $\id{p}\subset \Om$ be a prime
  ideal and let $\overline{k(\id{p})}$ be an algebraic closure of the residue
  field $k(\id{p})$. Denote $\specY_{\Gamma
    ((n),\id{a}),\overline{k(\id{p})}}=\specY_{\Gamma ((n),\id{a})}
  \underset {\ZZ[1/(Dn),\zeta_{n}]}{\times \overline{k(\id{p})}}$. If
  $Y_{\Gamma ((n),\id{a})}$ is a minimal
  surface of general type then the same is true for $\specY_{\Gamma
    ((n),\id{a}),\overline{k(\id{p})}}$.
\end{prop}
\begin{proof}
  This follows from \cite{Ueno} Theorem 9.1 and Lemma 9.6. We would
  like to thank Qing Liu by 
pointing us this result via mathoverflow.
\end{proof}
 
We now assume that $D$ and $n$ satisfy furthermore the hypothesis of
the  previous section and we
use the notations of that section. Again, for simplicity we state the
result for the infinity cusp. 

\begin{thm}[Sturm bound] \label{thm:6}
Let $\Om\subset \CC$ be
a ring of fractions of the ring of integers of a number field.
Let $G$ be a Hilbert modular form of parallel weight $2k$ for
$\Gamma(\Om_{K},\id{a})$, which vanishes with order $s$ at all
cusps. Suppose that the Fourier expansion of $G$ at the infinity cusp
$c_{1}$ is 
\[ 
G= \sum_{\xi \in (\id{a})^\vee_+\cup\{0\}} a_\xi \exp(\xi
z_1 + \xi' z_2), 
\]
with $a_{\xi}\in \Om$ for all $\xi \in M^\vee_+\cup\{0\}$.
Let $\id{p}\subset \Om$ be a prime ideal such that $\id{p} \nmid Dn$
and let $a$ be an integer with
\[
a > \frac{4kn \zeta_K(-1)}{\sum_j (b_{1,j}-2)}-s
\left(\frac{\sum_{i=1}^h \sum_j
    (b_{i,j}-2)}{\sum_j(b_{1,j}-2)}\right).
\] 

If $a_{\xi}\in \id{p}$ for all $\xi\in (\id{a})^\vee_+\cup\{0\}$ such
that there is a $j\in J$ with $\Tr(\xi A_{j})< a+s$, then $a_\xi
\in\id{p}$ for all $\xi\in M^\vee_+\cup\{0\}$.  
\end{thm}

\begin{proof}
  With the same argument as in the proof of the classical Sturm theorem, we
  can assume without loss of generality that
  $\ZZ[1/(Dn),\zeta_{n}]\subset \Om$. We consider the regular model
  $\specY_{\Gamma
    ((n),\id{a})}$  of $Y_{\Gamma
    ((n),\id{a})}$ provided by Theorem \ref{thm:3}. As before, we
  denote by $\specY_{\Gamma
    ((n),\id{a}),\Om}$ the model over $\Spec(\Om)$ obtained after base
  change.  
  Since $G$ is a modular form for $\Gamma(\Om_{K},\id{a})$ it is also
  a modular form for $\Gamma((n),\id{a})$.
By the $q$-expansion principle
  (Theorem \ref{thm:5}) the modular form $G$ determines a section of 
$\Om_{\specY_{\Gamma ((n),\id{a}),\Om}}(k(\bfK_{\Om}+\Sc'_{\Om}))$,
that we denote also by $G$. The
vanishing hypothesis imply that, when we restrict $G$
to $\specY_{\Gamma ((n),\id{a}),\overline {k(\id{p})}}$ we
obtain a global section of
\begin{displaymath}
  \Om_{\specY_{\Gamma
      ((n),\id{a}),\overline{k(\id{p})}}}(k(\bfK_{\overline{k(\id{p})}}
  +\Sc'_{\overline{k(\id{p})}})-sn\Sc'_{\overline{k(\id{p})}}
  -an\Sc'_{i_{0},\overline{k(\id{p})}}). 
\end{displaymath}
By Proposition \ref{prop:3} the canonical divisor
$\bfK_{\overline{k(\id{p})}}$ is NEF. Since intersection numbers are
preserved by specialization, from equation \eqref{eq:7} we deduce that
\begin{displaymath}
  \bfK_{\overline{k(\id{p})}}\cdot(k(\bfK_{\overline{k(\id{p})}}
  +\Sc'_{\overline{k(\id{p})}})-sn\Sc'_{\overline{k(\id{p})}}
  -an\Sc'_{i_{0},\overline{k(\id{p})}})<0
\end{displaymath}
Therefore the  restriction of $G$ to $\specY_{\Gamma
  ((n),\id{a}),\overline {k(\id{p})}}$ is zero, proving the result.
\end{proof}

\section{General weights and levels.}

Although the main results of the previous sections are stated only for
modular forms of level $\Gamma(\Om_K,\id{a})$ and parallel weight
$(2k,2k)$, they can be generalized to any congruence subgroup $\Gamma_{\id{a}}$
and any weight $(k_1,k_2)$ satisfying the parity condition $k_1 \equiv k_2
\pmod{2}$ using exactly the same tricks as for classical modular
forms. Assume that $n$ satisfies the hypothesis of
Theorem~\ref{thm:1}.

  Let $\Gamma_{\id{a}}$ be a congruence subgroup,
  $(k_1,k_2)$ a weight satisfying the previous parity condition.
  Let $\{A_{j}\}_{j\in J}$ be a set of representatives under the action
  of $U^{2}_{\Om_K}$, of the corners of
  the convex hull of $(\id{a}^{-1})_{+}$ as in Section
  \ref{sec:cusp-resolution}. 

\begin{thm}
  Let $G$ be a modular form of weight $(k_1,k_2)$ for $\Gamma_{\id{a}}$ which
  vanishes with order $s$ at all the cusps. Suppose that the Fourier
  expansion of $G$ at the infinity cusp is
\[
  G= \sum_{\xi \in M^{\vee}_+ \cup\{0\}} a_\xi \exp(2\pi i(\xi
z_1 + \xi' z_2)).
\]
for an appropriate lattice $M\subset \id{a}^{-1}$.
Let
\[
a > \frac{(k_1+k_2)n[\Gamma(\Om_K,\id{a}):\Gamma_{\id{a}}] \zeta_K(-1)}{\sum_j
  (b_{1,j}-2)}-s
\left(\frac{\sum_{i=1}^h \sum_j
    (b_{i,j}-2)}{\sum_j(b_{1,j}-2)}\right)
\] 
be an integer. 
If $a_{\xi}=0$ for all $\xi\in M^{\vee}_+ \cup\{0\}$ such that there
is a $j\in J$ with $\Tr(\xi A_{j})< a+s$, then $G=0$.
\end{thm}
\begin{proof}
Assume first that $k_1=k_2=2k$. Let $H(z_1,z_2)$ be the Hilbert
modular form given by
\[
H(z_1,z_2) = \prod_{\substack{\alpha \in \Gamma\backslash \Gamma(\Om_K,\id{a})\\ \alpha \not \in \Gamma}}G(z_1,z_2)|_{2k}[\alpha],
\]
where the product is taken over coset representatives of
$\Gamma(\Om_K,\id{a})$ modulo $\Gamma_{\id{a}}$ (acting on the left) not in the
trivial class. 

The form $G(z_1,z_2)H(z_1,z_2)$ is a form of weight
$2k[\Gamma(\Om_K,\id{a}):\Gamma_{\id{a}}]$ for $\Gamma(\Om_K,\id{a})$,
so we can apply the Hecke bound of section \ref{sec-Hecke} to
it. There is an integer $N$ such that $\Gamma ((N),\id{a})\subset
\Gamma _{\id{a}}$ and $N(\id{a}^{-1})\subset M$, thus
we can write  
the Fourier expansion of
$G$ as
\[
G(z_1,z_2)=\sum_{\xi \in
  \frac{1}{N}(\id{a}^{-1})^\vee_+\cup\{0\}}a_\xi\exp(2 \pi i (\xi
z_1+\xi'z_2)).
\]
Since $\Gamma((N),\id{a})$ is a normal subgroup of
$\Gamma(\Om_K,\id{a})$, the function $H(z_1,z_2)$ is a modular form
for it. Thus it has a Fourier expansion
\[
H(z_1,z_2)=\sum_{\xi \in \frac{1}{N}(\id{a}^{-1})^\vee_+\cup\{0\}}b_\xi\exp(2 \pi i (\xi z_1+\xi'z_2)).
\]
The product of this two Fourier expansions is
\[
\sum_{\eta \in (\id{a}^{-1})^\vee_+}\left(\sum_{\xi,\eta-\xi \in \frac{1}{N}(\id{a}^{-1})^\vee_+\cup\{0\}}a_\xi b_{\eta-\xi}\right)\exp(2 \pi i (\eta z_1+\eta'z_2)).
\]
In principle, the exterior sum should run over elements in
$\frac{1}{N}(\id{a})^\vee_+$, but since we know that $GH$ is a modular 
form for $\Gamma(\Om_K,\id{a})$, all the other terms are zero.

Note that since $\eta-\xi \gg 0$ (or zero), $\eta -\xi \ge 0$ and
$\eta'-\xi' \ge 0$, so $\Tr(\xi m) \le \Tr(\eta m)$ for $m \in
\id{a}^{-1}_+$. In particular, if $a_\xi =0$ for all the elements in
the hypothesis, the coefficients of $G(z_1,z_2)H(z_1,z_2)$ are all zero for all
$\eta$ with $\Tr(\eta A_{j})\le a+s$ for some $j\in J$ and the result
follows from Corollary \ref{coro:mainHecke}.

For general weights $(k_1,k_2)$, it is enough to apply the previous
case to the form $G(z_1,z_2)
G(z_2,z_1)$, which has parallel weight $k_1+k_2$ (even) and vanishes
with order $2s$ at all the cusps and with order $2a+2s$ at the
infinity cusp.
\end{proof}

\begin{remark}
  A similar Sturm bound holds for general weights and level, we leave
  it as an exercise.
\end{remark}

\begin{remark}
  As in the classical case, one can obtain for forms in
  $\Gamma_0(\id{c},\id{a},\chi)$ (i.e. forms with a character) the
  same bound as the one for the subgroup $\Gamma_0(\id{c},\id{a})$, by
  using Buzzard's trick. If $\ord(\chi)$ denotes the order of $\chi$, then we
  consider $G(z_1,z_2)^{\ord(\chi)}$, which vanishes with order $\ord(\chi)s$ at all cusps
  and $\ord(\chi)s+\ord(\chi)a$ at the infinity cusp, but is a form for $\Gamma_0(\id{c},\id{a})$, so
  the values of $\ord(\chi)$ cancels in the formula.
\end{remark}

\begin{remark}
  If in the Hecke/Sturm bound we fix the level and let the weight
  grow, the number of elements of the  Fourier expansion to check
  equality/congruence grows quadratically with the weight since we
  have to search for elements in a cone whose trace grows linearly in
  the weight. If we stick to parallel weight forms, it is known that
  the same happens with the dimension of such modular forms spaces.
  This implies that the bound we got is the best possible up to a
  constant (depending only on the level and the base field).
\end{remark}

\begin{remark}
  When the narrow class number is greater than $1$, one can relate
  modular forms for the different subgroups $\PGL_2^+(\Om_K,\id{a})$
  (varying $\id{a}$) using the action of the Hecke operators. This
  allows to take the number of coefficients needed to check
  congruences/equality of modular forms to be the minimum between all the ideals, but they need not be the ones with smaller trace. See the Remark~\ref{rem:p11}.
\end{remark}

\section{Examples}

\subsection{The case $\Q(\sqrt{10})$} This is the first real
quadratic field with non-trivial class group. The class group has
order $2$ and the two representatives can be taken as $1$ and
$\<2,\sqrt{10}>$ (the unique prime ideal dividing $2$). The
discriminant of such field is $D=40 \not \equiv 1 \pmod 8$, hence
Conjecture \ref{conj:canonicaldivisor} holds and we can take $n=3$ for
the Hecke/Sturm bound. Applying the desingularization process of
Appendix \ref{appendix:a}, we see that for the principal ideal the
picture looks like Table~\ref{fig:example2_1}.
\begin{table}[h]
\begin{multicols}{2}
\psfrag{a}{\tiny $m_1$}
\psfrag{b}{\tiny $m_2$}
\psfrag{c}{\tiny $m_3$}
\psfrag{d}{\tiny $m_4$}
\psfrag{e}{\tiny $m_5$}
\psfrag{f}{\tiny $m_6$}
        \includegraphics[width=3cm]{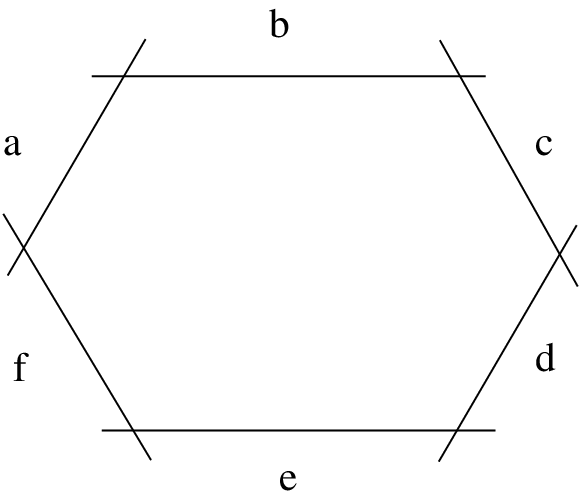}

\scalebox{0.9}{
\begin{tabular}{|c|c|c|}
\hline 
Label & Point & S.I.\\% & Label & Point & S.I.\\
\hline \hline
$m_1$ & $1$ & $-8$ \\
\hline
 $m_2$ & $4-\sqrt{10}$ & $-2$\\
\hline
$m_3$ & $7-2\sqrt{10}$ & $-2$ \\ 
\hline
$m_4$& $10-3\sqrt{10}$ & $-2$\\
\hline
$m_5$& $13-4\sqrt{10}$ & $-2$\\
\hline
$m_6$& $16-5\sqrt{10}$ & $-2$\\
\hline
\end{tabular}}
\end{multicols}
\caption{Infinity cusp desingularization for $\Gamma(\Om_{\Q(\sqrt{10})},1)$}
\label{fig:example2_1}
\end{table}

The bound then reads for the infinity cusp 
\[
a> \frac{2\cdot 2k\cdot3\cdot 7}{6\cdot6}-s\frac{6+4}{6} =\frac{7k-2s}{3}-s.
\]
If $G(z_1,z_2) \in M_{2k}(\SL_2(\Om_K))$, we have
\[
G(z_1,z_2)=\sum_{\xi \in (\frac{1}{2}\ZZ+\frac{1}{2\sqrt{10}}\ZZ)^+}
a_\xi \exp(2\pi i(\xi z_1 + \xi' z_2)).\] If $a_\xi = 0$ for all $\xi$
with $\trace(m\xi)\le \frac{7k-2s}{3}$, with $m$ any of the six
vertexes then $G(z_1,z_2)$ is the zero form. In particular, for cusp
forms of parallel weight $2$, whose dimension is $1$, we only need to
check the elements with trace one. The first vertex gives the
non-equivalent points
\[
\xi=\frac{-2}{2\sqrt{10}} + \frac{1}{2},\frac{-1}{2\sqrt{10}} + \frac{1}{2},\frac{1}{2},\frac{1}{2\sqrt{10}} + \frac{1}{2},\frac{2} {2\sqrt{10}} + \frac{1}{2},\frac{3}{2\sqrt{10}} + \frac{1}{2}.
\]

All the other ones give the point
$\xi=\frac{3}{2\sqrt{10}}+\frac{1}{2}$.

Here is a small table comparing the number of elements and the dimensions for some values of $k$:

\begin{table}[h]
\begin{tabular}{|c|c|c|c|c|c|c|}
\hline 
$2k$ & $20$ &$30$ &$40$ &$50$ &$100$&$150$\\
\hline 
Number of Elts & $1518$ & $3570$ & $6486$ & $9918$ & $40716$&$91350$ \\
\hline
Dimension & $212$ & $492$ & $888$ & $1402$ & $5718$ &$12952$ \\
\hline
\end{tabular}
\end{table}

Looking at the other cusp corresponds to look at the infinity cusp for
the level $\<2,\sqrt{10}>$. For this level, the desingularization at
infinity looks like Table~\ref{fig:example2_2}.
\begin{table}[h]
\begin{multicols}{2}
\psfrag{A}{\tiny $m_1$}
\psfrag{B}{\tiny $m_2$}
\psfrag{C}{\tiny $m_3$}
\psfrag{D}{\tiny $m_4$}
        \includegraphics[width=3cm]{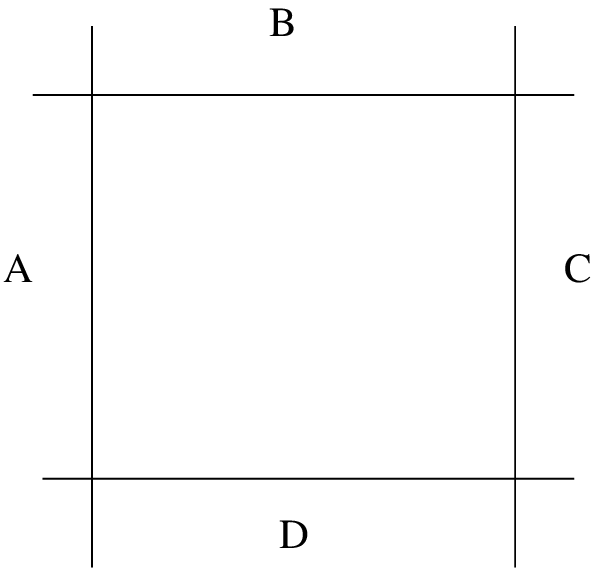}

\scalebox{0.9}{
\begin{tabular}{|c|c|c|}
\hline 
Label & Point & S.I.\\% & Label & Point & S.I.\\
\hline \hline
$m_1$ & $2$ & $-4$ \\
\hline
 $m_2$ & $4-\sqrt{10}$ & $-3$\\
\hline
$m_3$ & $10-3\sqrt{10}$ & $-2$ \\ 
\hline
$m_4$& $16-5\sqrt{10}$ & $-3$\\
\hline
\end{tabular}}
\end{multicols}
\caption{Infinity cusp desingularization for $\Gamma(\Om_{\Q(\sqrt{10})},\<2,\sqrt{10}>)$}
\label{fig:example2_2}
\end{table}

Then, the bound for this level at the infinity cusp reads
\[
a> \frac{2\cdot 2k\cdot3\cdot 7}{4\cdot6}-s\frac{6+4}{4} =\frac{7k-3s}{2}-s.
\]
If $G(z_1,z_2) \in M_{2k}(\Gamma(\Om_K,\<2,\sqrt{10}>))$, we have
\[
G(z_1,z_2)=\sum_{\xi \in (\frac{1}{4}\ZZ+\frac{1}{2\sqrt{10}}\ZZ)^+}
a_\xi \exp(2\pi i(\xi z_1 + \xi' z_2)).\] If $a_\xi = 0$ for all $\xi$
with $\trace(m\xi)\le \frac{7k-3s}{2}$, with $m$ any of the four
vertexes then $G(z_1,z_2)$ is the zero form. For cusp forms of
parallel weight $2$, whose dimension is $1$, we need to check the
elements with trace one or two. The first vertex gives the
non-equivalent points (up to units squared)
\[
\xi=\frac{1}{4}, \frac{1}{4}+\frac{1}{2\sqrt{10}},\frac{1}{4}-\frac{1}{2\sqrt{10}},\frac{1}{2}-\frac{1}{\sqrt{10}},\frac{1}{2}-\frac{1}{2\sqrt{10}},\frac{1}{2},\frac{1}{2}+\frac{1}{\sqrt{10}},\frac{1}{2}+\frac{1}{2\sqrt{10}},\frac{1}{2}+\frac{3}{2\sqrt{10}}.
\]
The first three points have trace $1$, while the others trace $2$. The second vertex gives the points
\[
\xi=\frac{1}{4}+\frac{1}{2\sqrt{10}},\frac{1}{2}+\frac{3}{2\sqrt{10}},\frac{1}{4},\frac{1}{2}+\frac{1}{\sqrt{10}},\frac{3}{4}+\frac{2}{\sqrt{10}},1+\frac{3}{\sqrt{10}},
\]
where the first two elements give trace $1$ while the others trace $2$. The third vertex gives the points
\[
\xi=\frac{1}{2}+\frac{3}{2\sqrt{10}},\frac{1}{4}+\frac{1}{2\sqrt{10}},1+\frac{3}{\sqrt{10}},\frac{7}{4}+\frac{11}{2\sqrt{10}},
\]
where the first one corresponds to trace $1$ and the other to trace
$2$. Note that the last element is equivalent to
$\frac{1}{4}-\frac{1}{2\sqrt{10}}$. The last vertex gives the points
\[
\xi = \frac{1}{2}+\frac{3}{2\sqrt{10}},\frac{7}{4}+\frac{11}{2\sqrt{10}},1+\frac{3}{\sqrt{10}},\frac{9}{4}+\frac{7}{\sqrt{10}},\frac{7}{2}+\frac{11}{\sqrt{10}},\frac{19}{4}+\frac{15}{\sqrt{10}},
\]
where the first two elements correspond to trace $1$ and the others to trace $2$. The last two elements are equivalent to the elements $\frac{1}{2}-\frac{1}{\sqrt{10}}$ and $\frac{1}{4}$ respectively, so we need to check $12$ coefficients.

Here is a small table comparing the number of elements and the dimensions for some values of $k$:

\begin{table}[h]
\begin{tabular}{|c|c|c|c|c|c|c|}
\hline 
$2k$ & $20$ &$30$ &$40$ &$50$ &$100$&$150$ \\
\hline 
Number of Elts & $2244$ & $5304$ & $9384$ & $14964$ & $60204$&$135720$ \\
\hline
Dimension & $212$ & $492$ & $888$ & $1402$ & $5718$ & $12952$\\
\hline
\end{tabular}
\end{table}

\subsection{The case $\Q(\sqrt{29})$} In this case the class number
and the narrow class number are both one. The discriminant is $29 \not
\equiv 1 \pmod 8$, hence Conjecture \ref{conj:canonicaldivisor} holds
and we can take $n=3$ for the Hecke/Sturm bound. Applying the
desingularization process of Appendix \ref{appendix:a}, we see that
for the principal ideal the picture looks like
Table~\ref{fig:example3_1}.

\begin{table}[h]
\begin{multicols}{2}
\psfrag{a}{\tiny $m_1$}
\psfrag{b}{\tiny $m_2$}
\psfrag{c}{\tiny $m_3$}
\psfrag{d}{\tiny $m_4$}
\psfrag{e}{\tiny $m_5$}
        \includegraphics[width=3cm]{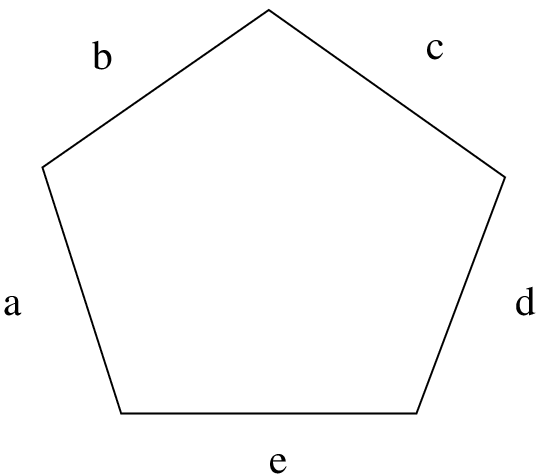}

\scalebox{0.9}{
\begin{tabular}{|c|c|c|}
\hline 
Label & Point & S.I.\\% & Label & Point & S.I.\\
\hline \hline
$m_1$ & $1$ & $-7$ \\
\hline
 $m_2$ & $\frac{7-\sqrt{29}}{2}$ & $-2$\\
\hline
$m_3$ & $6-\sqrt{29}$ & $-2$ \\ 
\hline
$m_4$& $\frac{17-3\sqrt{29}}{2}$ & $-2$\\
\hline
$m_5$& $11-2\sqrt{29}$ & $-2$\\
\hline
\end{tabular}}
\end{multicols}
\caption{Infinity cusp desingularization for $\Gamma(\Om_{\Q(\sqrt{29})},1)$}
\label{fig:example3_1}
\end{table}

Then, the bound for this level at the infinity cusp reads
\[
a> \frac{2\cdot 2k\cdot3\cdot 1}{5\cdot2}-s =\frac{6k}{5}-s.
\]
If $G(z_1,z_2) \in M_{2k}(\SL_2(\Om_K))$, we have
\[
G(z_1,z_2)=\sum_{\xi \in
  (\frac{1}{\sqrt{29}}\ZZ+(\frac{1}{2}+\frac{1}{2\sqrt{29}})\ZZ)^+}
a_\xi \exp(2\pi i(\xi z_1 + \xi' z_2)).\] If $a_\xi = 0$ for all $\xi$
with $\trace(m\xi)\le \frac{6k}{5}$, with $m$ any of the five vertexes
then $G(z_1,z_2)$ is the zero form. For cusp forms of parallel weight
$2$, whose dimension is $1$, we need to check the elements with trace
one. The first vertex gives the five non-equivalent points 
\[
\xi=\frac{1}{2}\pm\frac{1}{2\sqrt{29}},\frac{1}{2}\pm\frac{3}{2\sqrt{29}},\frac{1}{2}+\frac{5}{2\sqrt{29}}.
\]
The second vertex, the third vertex, the fourth and the fifth vertex
give the point $\xi=\frac{1}{2}+\frac{5}{2\sqrt{29}}$. So we need to
check $5$ elements of the Fourier expansion.

Here is a small table comparing the number of elements and the dimensions for some values of $k$:

\begin{table}[h]
\begin{tabular}{|c|c|c|c|c|c|c|c|c|}
\hline 
$2k$ & $20$ &$30$ &$40$ &$50$ &$100$&$150$&$200$&$300$ \\
\hline 
Number of Elts & $390$ & $855$ & $1500$ & $2326$ & $9151$&$20477$ &$36302$ & $81453$\\
\hline
Dimension & $92$ & $212$ & $381$ & $602$ & $2451$&$5552$ & $9902$ & $22352$\\
\hline
\end{tabular}
\end{table}

\subsection{The case $\Q(\sqrt{11})$} This real quadratic field has
class number is $1$ and narrow class number is $2$. Generators are
given by the principal ideal and the prime ideal $(\sqrt{11})^{-1}$. Since
$D=44 \not \equiv 1 \pmod 8$, Conjecture \ref{conj:canonicaldivisor}
holds and we can take $n=3$. Applying the desingularization process of
Appendix \ref{appendix:a}, we see that for the principal ideal the
picture looks like Table~\ref{fig:example14_1}.

\begin{table}[h]
\begin{multicols}{2}
\psfrag{a}{\tiny $m_1$}
\psfrag{b}{\tiny $m_2$}
\psfrag{c}{\tiny $m_3$}
\psfrag{d}{\tiny $m_4$}
\psfrag{e}{\tiny $m_5$}
\psfrag{f}{\tiny $m_6$}
        \includegraphics[width=3cm]{Example10_1.eps}

\scalebox{0.9}{
\begin{tabular}{|c|c|c|}
\hline 
Label & Point & S.I.\\% & Label & Point & S.I.\\
\hline \hline
$m_1$ & $1$ & $-8$ \\
\hline
 $m_2$ & $4-\sqrt{11}$ & $-2$\\
\hline
$m_3$ & $7-2\sqrt{11}$ & $-2$ \\ 
\hline
$m_4$& $10-3\sqrt{11}$ & $-8$\\
\hline
$m_5$& $73-22\sqrt{11}$ & $-2$\\
\hline
$m_6$& $136-41\sqrt{11}$ & $-2$\\
\hline
\end{tabular}}
\end{multicols}
\caption{Infinity cusp desingularization for $\Gamma(\Om_{\Q(\sqrt{11})},1)$}
\label{fig:example14_1}
\end{table}

The bound then reads for the infinity cusp 
\[
a> \frac{4k\cdot3\cdot7}{12\cdot6}-s=\frac{7k}{6}-s.
\]
If $G(z_1,z_2) \in M_{2k}(\SL_2(\Om_K))$, we have
\[
G(z_1,z_2)=\sum_{\xi \in (\frac{1}{2}\ZZ+\frac{1}{2\sqrt{11}}\ZZ)^+}
a_\xi \exp(2\pi i(\xi z_1 + \xi' z_2)).\] If $a_\xi = 0$ for all $\xi$
with $\trace(m\xi)\le \frac{7k}{6}$, with $m$ any of the six
vertexes, then $G(z_1,z_2)$ is the zero form. For cusp forms of
parallel weight $2$, whose dimension is $2$, we need to check elements
with trace $1$. The first vertex gives the seven non-equivalent points
\[
\frac{1}{2}\pm\frac{3}{2 \sqrt{11}} ,\frac{1}{2}\pm\frac{1}{\sqrt{11}} ,\frac{1}{2}\pm\frac{1}{2 \sqrt{11}},\frac{1}{2}.
\]
The second and the third vertices give the point
$\frac{1}{2}+\frac{3}{2\sqrt{11}}$. The fourth vertex gives the points
\[
\frac{1}{2}+\frac{3}{2 \sqrt{11}}, 2+\frac{13}{2 \sqrt{11}}, \frac{7}{2}+\frac{23}{2 \sqrt{11}}, 5+\frac{33}{2 \sqrt{11}}, \frac{13}{2}+\frac{43}{2 \sqrt{11}}, 8+\frac{53}{2 \sqrt{11}}, \frac{19}{2}+\frac{63}{2 \sqrt{11}}.
\]
The last two vertices give the point
$\frac{19}{2}+\frac{63}{2\sqrt{11}}$. We have to check $12$
conditions, since the elements $\frac{19}{2}+\frac{63}{2\sqrt{11}}$ and
  $\frac{1}{2}-\frac{3}{2\sqrt{11}}$ differ by an even power of the
  fundamental unit.

Here is a small table comparing the number of elements and the dimensions for some values of $k$:

\begin{table}[h]
\begin{tabular}{|c|c|c|c|c|c|c|}
\hline 
$2k$ & $20$ &$30$ &$40$ &$50$ &$100$&$150$\\
\hline 
Number of Elts &$792$ & $1836$ & $3312$ & $5220$ & $20532$ & $45936$\\
\hline
Dimension &  $212$ & $492$ & $888$ & $1402$ & $5718$&$12952$\\

\hline
\end{tabular}
\end{table}

The desingularization for the class of $(\sqrt{11})^{-1}$ has $12$ lines. The representatives and their intersection is given in Table~\ref{table:case_14_2}.

\begin{table}[h]
\begin{tabular}{|c|c||c|c|}
\hline 
Point & S.I.&Point & S.I.\\% & Label & Point & S.I.\\
\hline \hline
 $\frac{1}{11}$ & $-5$ &$-\frac{3}{\sqrt{11}} + \frac{10}{11}$&$-5$\\
\hline
$-\frac{1}{2\sqrt{11}} + \frac{5}{22}$ &$-2$ &$-\frac{25}{2 \sqrt{11}} + \frac{83}{22}$&$-2$\\
\hline
$-\frac{1}{\sqrt{11}} + \frac{4}{11}$&$-2$ &$-\frac{22}{\sqrt{11}} + \frac{73}{11}$&$-2$\\
\hline
$-\frac{3}{2 \sqrt{11}} + \frac{1}{2}$ &$-2$ &$-\frac{63}{2\sqrt{11}} + \frac{19}{2}$&$-2$\\
\hline
$-\frac{2}{\sqrt{11}} + \frac{7}{11}$&$-2$ &$-\frac{41}{\sqrt{11}} + \frac{136}{11}$&$-2$\\
\hline
$-\frac{5}{2 \sqrt{11}} + \frac{17}{22}$&$-2$ &$-\frac{101}{2 \sqrt{11}} + \frac{335}{22}$&$-2$\\

\hline
\end{tabular}
\caption{Infinity cusp desingularization for $\Gamma(\Om_{\Q(\sqrt{11})},(\sqrt{11})^{-1})$}
\label{table:case_14_2}
\end{table}

Then, the bound for this level at the infinity cusp reads
\[
a> \frac{7k}{3}-s.
\]
If $G(z_1,z_2) \in M_{2k}(\Gamma(\Om_K,(\sqrt{11})^{-1}))$, we have

\[
G(z_1,z_2)=\sum_{\xi \in (\frac{1}{2}\ZZ+\frac{\sqrt{11}}{2}\ZZ)^+} a_\xi \exp(2\pi i(\xi
z_1 + \xi' z_2)).\]

If $a_\xi = 0$ for all $\xi$ with $\trace(m\xi)\le \frac{7k}{3}$, with
$m$ any of the twelve vertices then $G(z_1,z_2)$ is the zero form. For
$k=2$, the dimension is $3$ and we have to check $179$ conditions
(which correspond to $138$ different ideals).

Here is a small table comparing the number of elements and the dimensions for some values of $k$ (the dashes in the table mean the number could not be computed):

\begin{table}[h]
\begin{tabular}{|c|c|c|c|c|c|c|}
\hline 
$2k$ & $20$ &$30$ &$40$ &$50$ &$100$&$150$\\
\hline 
Number of Elts & $21483$ & $49585$ & -- & -- & -- & --\\
\hline
Dimension & $213$ & $493$ & $889$ & $1403$ & $5719$ &$12953$\\
\hline
\end{tabular}
\end{table}

\begin{remark}

  Hecke operators do not act on the surface $Y_{\Gamma(\Om_K,\Om_K)}$, but
  rather act as correspondences on the product of the surfaces
  $Y_{\PGL_2^+(\Om_K,\Om_K)}\times Y_{\PGL_2^+(\Om_K,(\sqrt{11})^{-1})}$,
  i.e. they act on product of Hilbert modular forms where the first
  component is invariant under $\PGL_2^+(\Om_K,1)$ and the second one
  under $\PGL_2^+(\Om_K,(\sqrt{11})^{-1})$ (this are the automorphic
  forms, see \cite{Garrett} for definitions of Hilbert modular forms,
  its relation with automorphic forms and Hecke operators). A form in
  $M_{\bf k}(\PGL_2^+(\Om_K,(\sqrt{11})^{-1}))$ can be thought as an
  automorphic form supported only in one component.

  Let $\id{p}_{11}$ denote the prime ideal generated by
  $\sqrt{11}$. Then the Hecke operator $T_{\id{p}_{11}}$ sends a form
  $F$ supported in $M_{\bf k}(\PGL_2^+(\Om_K,(\sqrt{11})^{-1}))$ to the
  form supported in $M_{\bf k}(\PGL_2^+(\Om_K,\Om_K))$. Furthermore, if 
\[
F(z_1,z_2) = \sum_{\xi \in M^\vee_+ \cup \{0\}} a_{\xi} \exp(2 \pi i(\xi z_1+\xi' z_2)).
\]
then 
\[
T_{\id{p}_{11}}(F)(z_1,z_2) = \sum_{\xi \in M^\vee_+ \cup \{0\}} \left(11a_{\xi}+a_{\frac{\xi}{11}}\right)\exp(2 \pi i(\xi z_1+\xi' z_2)).
\]
Assume that the form $F(z_1,z_2)$ is not in the kernel of the Hecke operator
$T_{\id{p}_{11}}$ (they are usually invertible operators). Then if the
Fourier coefficients $a_\xi$ and $a_{\frac{\xi}{11}}$, with $\xi$ in
the Hecke/Sturm set for the trivial class are all zero/congruent to
zero, then the form $F(z_1,z_2)$ itself is the zero form. This implies looking at
less than $4$ times the dimension coefficients instead of $100$ times
the dimension!

It is worthwile studying the action of the Hecke operators to improve
our Sturm bound for general real quadratic fields.
\label{rem:p11}  
\end{remark}

\appendix
\section{Desingularization algorithm}
\label{appendix:a}

Recall that the isotropy group of any cusp for $\Gamma(\id{c},\id{a})$
is conjugate to a group of the form $G(M,V)$, where $M\subset K$ is an
$\Om_K$-module and $V \subset U_K^+$ is a subgroup of finite index. As
a transformation group $G(M,V) = M \rtimes V$. To compute the
desingularization of the cusp we first look at the module $M$.

An oriented basis of $M$ is a $\ZZ$-basis $M =\<\alpha,\beta>$ such
that $\det\left(\begin{smallmatrix}\alpha & \beta \\ \alpha' &
    \beta'\end{smallmatrix}\right) > 0$. To an oriented basis we can
associate the indefinite binary quadratic form $Q(x,y) =
\frac{1}{\norm(M)}\norm(\alpha x + \beta y)$, where $\norm(M)$
indicates the content of the form $\norm(\alpha x + \beta y)$,
i.e. the rational number which makes $Q(x,y)$ an integral
primitive form. 

If $\lambda$ is a totally positive element, multiplication by
$\lambda$ sends oriented bases of $M$ to oriented bases of $\lambda
M$, but clearly $\<\alpha,\beta>$ and $\<\lambda \alpha,\lambda
\beta>$ have the same quadratic form attached. Choosing a different
oriented basis gives an $\SL_2(\ZZ)$-equivalent form, hence we get a
bijection between the narrow class group of $K$ and
$\SL_2(\ZZ)$-equivalence classes of integral primitive indefinite
binary quadratic forms of discriminant $D$.

Following \cite{VDG}, we call a form $ax^2+bxy+cy^2$ of discriminant
$D$ \emph{reduced} if
\begin{equation}
0< \frac{b - \sqrt{D}}{2a} <1< \frac{b + \sqrt{D}}{2a}.
\label{eq:reduced}
\end{equation}

Using strict $\SL_2(\ZZ)$ equivalence, one can reduce any indefinite
integral binary quadratic form of discriminant $D$ to a reduced
one. In other words, starting from $M$ one gets an oriented basis of
the form $\lambda M =\left(\frac{b+\sqrt{D}}{2a}\right)\ZZ+\ZZ$.

\begin{remark}
  This notion of a reduced form is not universal. For example, in
  Cohen's book (see (\cite{cohen} Definition 5.6.2) a reduced
  indefinite integral binary quadratic form satisfies
\[
0 \le \frac{\sqrt{D}-b}{2|a|} <1<\frac{\sqrt{D}+b}{2|a|}.
\]
Starting from $Q(x,y)$ one can use Cohen's algorithm (\cite{cohen}
Algorithm 5.6.5 which is for example implemented in \cite{PARI2}) to
get a Cohen-reduced form. Note that we can always take as reduced form one
with $a>0$ (by Proposition 5.6.6 of \cite{cohen}) and remove the
previous absolute value. If we apply the change of variables given by
the matrix $\left(\begin{smallmatrix}
    1&1\\0&1\end{smallmatrix}\right)$, which sends $b$ to $b+2a$, we
get a reduced form in the sence of \eqref{eq:reduced}.
\end{remark}

Once we computed the reduced basis for $\lambda M$, the first vertices
of the convex hull are:
\begin{equation*}
A_{-1} = w_0:=\frac{b+\sqrt{D}}{2a}, \qquad \qquad A_0 = 1, \qquad \qquad A_{k+1}:= b_kA_k-A_{k-1},
\end{equation*}
where the numbers $b_k$ are defined recursively by 
\[
b_k := \lceil w_k \rceil \qquad   \text{ and }\qquad w_{k+1} := \frac{1}{b_k -
  w_k}.
\]

Now we add the multiplicative structure. Let $\varepsilon$ be a
generator of $U_{\Om_K}^2$. It acts on the sequence $\{A_k\}$ with a
finite number of representatives. Moreover, the sequence $\{b_k\}$
(and $\{w_k\}$ also) is periodic with some length $r$. Then for all $k \in \ZZ$,
\[
A_k=\varepsilon^{\pm\nu}A_{k+r},
\]
where $\nu =1$ if the fundamental unit of $K$ has norm $-1$ and $\nu=2$ otherwise.

Let $\tilde{r} = r\cdot \nu \cdot [U_{\Om_K}^2 :V]$.  Then the
resolution attached to $G(M,V)$ consists of $\tilde{r}$ lines $S_k$,
$k \in \ZZ/\tilde{r}$ (each one isomorphic to $\PP^1$) which satisfy:
\begin{itemize}
\item $S_k^2 = -b_k$ if $\tilde{r} \ge 2$.
\item Let $n,m$ be integers and $\tilde{r} \ge 3$. Then:
\begin{itemize}
\item If $n \not \equiv m \pm 1 \pmod{\tilde{r}}$, $S_n \cap  S_m = \emptyset$.
\item If $n \equiv \pm 1 \pmod{\tilde{r}}$, $S_n \cap S_m$ is one point.
\end{itemize}
\item If $\tilde{r}=1$, then $S_0$ is singular and $S_0^2=-b_0+2$.
\item If $\tilde{r}=2$, then $S_0$ and $S_1$ are non-singular and intersect in
  $2$ points.
\end{itemize}

\section{Rational case}
\label{appendix:b}

Recall that $Y_{\Gamma(\Om_K,\id{a})}$ is rational for $D=5, 8, 12,
13, 17, 21, 24, 28, 33, 60$ and $\id{a}$ in the principal genus, or
for $D=12$ and $\id{a}$ not in the principal genus. The purpose of this
appendix is to give a Sturm bound for some of these cases.  If $\id{c}$ is an
integral ideal such that $Y_{\Gamma(\id{c},\id{a})}$ or a blow
down of it, is a minimal surface of general type we still get the
Hecke/Sturm bound
\[
a > \frac{4k[\Gamma(\id{c},\id{a}):\Gamma(\Om_K,\id{a})]\zeta_k(-1)}{\sum_j (b_{i_{0},j}-2)}-s
\left(\frac{\sum_{i=1}^h \sum_j
    (b_{i,j}-2)}{\sum_j(b_{i_{0},j}-2)}\right),
\]
where the numbers $b_{i,j}$ are the ones appearing in the cusp
desingularization process of $Y_{\Gamma(\id{c},\id{a})}$, or that of
its blow down. Here is a summary of the ideals $\id{c}$ which give a
minimal surface of general type for some values of $D$:
\begin{itemize}
\item $D=5$: $\id{c}=3$ (\cite{VDG} Example 7.5 p. 179). There are ten
  non-equivalent cusps, each one resolved by a cycle $(3,3,3,3)$.
\item $D=8$: $\id{c}=\id{p}_7$ a prime ideal or norm $7$ (\cite{VDG}
  page 196). There are eight cusps, each one resolved by a cycle
  $(4,2,4,2,4,2)$.
%\item If $D=12$, let $\id{p}_2$ be
%  a prime ideal of norm $2$, then $Y_{\Gamma(\id{p}_2,\Om_K)}$ is a blown-up K3
%  surface. TERMINAR
%so Conjecure~\ref{conj:canonicaldivisor} holds and we can
%  apply Theorem~\ref{thm:bestbound}.
\item $D=13$, $\id{c}=2$ (see \cite{ZVDG} page 197) gives a surface of
  general type with the components of $F_1$ as the unique exceptional
  curves. There are $5$ cusps, each one in the minimal model is
  resolved by a cycle $(2,2,3,2,2,3,2,2,3)$ .
\item $D=17$: $\id{c}=2$ (see \cite{VDG}, page 198) gives a surface of
  general type with the components of $F_1$ as the unique exceptional
  curves. There are $9$ cusps, each one resolved in the minimal
  model by a cycle $(2,2,3,3,3)$ .
\item $D=21$: $\id{c}=2$ (Theorem 3 of \cite{ZVDG}) gives a surface of
  general type with the components of $F_1$ as the unique exceptional
  curves. There are $5$ cusps, each one resolved in the minimal model
  by a cycle $(5,5,5,5,5,5)$ .
\item If $D=24$: $\id{c}=\id{p}_2$, the prime ideal of norm $2$ (see
  \cite{VDG2} page 166) gives a surface of general type with the
  components of $F_1$ as the unique exceptional ones. There are $3$
  non-equivalent cusps, each one resolved in the minimal model by
  a cycle $(2,2,2,3,2,2,2,3)$ .
%\item If $D=28$,
%\item If $D=33$,
%\item If $D=60$,
\item $D=12$ and $\id{a}$ not in the principal genus: $\id{c}=2$
  (\cite{VDG} page 197) gives a minimal surface of general type. There
  are $3$ cusps, each one resolved by a cycle $(2,3)$.
\end{itemize}

%\begin{remark}
%  For the case $D=24$ and $\id{a}$ not in the principal genus, one can
%  take $\id{p}_3$ a prime ideal of norm $3$ and get a small
%  improvement in the bound than that of section ??. See \cite{VDG} for
%  the numerical invariants in this case.
%\end{remark}

With these data, we get the following Hecke bounds for Hilbert modular
form of parallel weight $k$, level $\Gamma(\id{c},\id{a})$ and
vanishing with order $s$ at all cusps:

\begin{table}[h]
\begin{tabular}{|c|c|c|c|c|c|c|c|c|}
\hline 
$D$ & $5$ &$8$ &$12$&$13$&$17$&$21$&$24$\\
\hline 
$\id{a}$ & $1$ & $1$ &$\sqrt{3}$&$1$&$1$&$1$&$1$\\
\hline
$a>$ & $48k-10s$ & $\frac{14k}{3}-8s$ &$4k-3s$&$\frac{40k}{3}-5s$&$4k-9s$&$\frac{40k}{9}-5s$&$12k-3s$\\
\hline
\end{tabular}
\end{table}

%FALTAN: D=12, \id{a}=1, D=28, 33, 60

\bibliographystyle{alpha}
\bibliography{Sturm}

\end{document}